\documentclass{article}
\usepackage[left=3.0cm,right=3.0cm,top=3.0cm,bottom=3.0cm]{geometry}
\usepackage{amsfonts}
\usepackage{graphicx}
\usepackage{epstopdf}
\usepackage{algorithmic}
\usepackage[numbers,sort]{natbib}
\usepackage{amsopn}

\usepackage{amssymb, amsthm}
\usepackage{mathtools}
\usepackage{mathrsfs}
\usepackage{hyperref}
\usepackage{enumitem}
\usepackage{subfig}
\usepackage{bbm}
\usepackage{comment}
\usepackage{upgreek}
\usepackage{cleveref}

\ifpdf
  \DeclareGraphicsExtensions{.eps,.pdf,.png,.jpg}
\else
  \DeclareGraphicsExtensions{.eps}
\fi

\title{Sample-Based Consistency in Infinite-Dimensional Conic-Constrained Stochastic Optimization}

\author{Caroline Geiersbach\thanks{University of Klagenfurt, Universit\"atstra{\ss}e 65-67, 9020 Klagenfurt, Austria (caroline.geiersbach@aau.at)}
\and Johannes Milz\thanks{H. Milton Stewart School of Industrial and Systems Engineering, Georgia Institute of Technology,
Atlanta, GA 30318 USA
(johannes.milz@isye.gatech.edu)
}}

\numberwithin{equation}{section}
\setenumerate{label=\textup{(\alph*)}}

\crefdefaultlabelformat{#2\textup{#1}#3}

\newcommand{\embedding}{\hookrightarrow}
\newcommand{\domain}{D}
\newcommand{\du}{\, \ensuremath{\mathrm{d}}}
\newcommand{\pP}{\mathbb{P}}
\newcommand{\E}{\mathbb{E}}
\newcommand{\real}{\mathbb{R}}

\newcommand{\wpone}{w.p.$1$}
\newcommand{\dom}[1]{\mathrm{dom}(#1)}

\newcommand{\D}{\du}
\newcommand{\R}{\mathbb{R}}
\newcommand{\Uad}{U_{\textup{ad}}}

\newtheorem{theorem}{Theorem}[section]
\newtheorem{proposition}[theorem]{Proposition}
\newtheorem{remark}[theorem]{Remark}
\newtheorem{lemma}[theorem]{Lemma}
\newtheorem{corollary}[theorem]{Corollary}
\newcounter{assumption}
\newtheorem{assumption}[theorem]{Assumption}
\crefname{assumption}{Assumption}{Assumptions}
\newtheorem{example}[theorem]{Example}
\crefname{example}{Example}{Examples}

\makeatletter
  \@mparswitchfalse
\makeatother

\begin{document}

\maketitle

\begin{abstract}
This paper is concerned with a class of stochastic optimization problems defined on a Banach space with almost sure conic-type constraints. For this class of problems, we investigate the consistency of optimal values and solutions corresponding to sample average approximation. Consistency is also shown in the case where a Moreau--Yosida-type regularization of the constraint is used.
Additionally, the consistency of Karush--Kuhn--Tucker conditions is shown under mild conditions. This work provides theoretical justification for the numerical computation of solutions frequently used in the literature. 
Several applications are explored showing the flexibility of the framework.
We cover nonparametric regression over Sobolev balls, operator learning,
optimal transport,  optimization with dynamical systems under uncertainty, and
optimization with partial differential equations under uncertainty.
\end{abstract}



\section{Introduction}
Many applications in optimization contain parameters or inputs that are not exactly known but for which a reasonable model of the distribution is available. This is the subject of stochastic programming, which is an established yet continually developing area of research. 
This paper is motivated by infinite-dimensional stochastic optimization
problems with uncertain conic constraints that must hold for
almost every (a.e.) realization of the uncertainty. Two concrete instances are
as follows. In both examples, $D\subset\mathbb{R}^d$ denotes a bounded
Lipschitz domain.%

\begin{example}[{Sobolev regression with pointwise nonnegativity}]
\label{example:regression}
Our first example is nonparametric regression with a pointwise
nonnegativity constraint, motivated by \cite[Example~3]{Cucker2002}. We seek
to learn an unknown nonnegative function on a bounded domain from data by
least squares, while restricting the search to a bounded Sobolev ball of
sufficiently high order so that point evaluations are well-defined.

More formally, 
let $u\in H^s(\domain)$ be a regression function with integer $s>d/2$, and let $\xi=(
x,y)$ be a random
vector with support  $\bar\domain\times[-1,1]$. Writing
$Bu\in \mathcal{C}(\bar\domain)$ for the continuous representative of $u$,
where $B\colon H^s(\domain)\to\mathcal{C}(\bar\domain)$ is the Sobolev embedding,
we consider the risk-neutral problem
\begin{align*}
\min_{u \in H^s(\domain), \, \|u\|_{H^s(\domain)}\le \rho}\,
    \tfrac{1}{2}\mathbb{E}[([Bu](x)-y)^2]
    \;\; \text{s.t.} \;\; [Bu](x) \geq 0 \;\; \text{a.s.},
\end{align*}
where $\rho >0$.
Here, the conic constraint is equivalent to
$[Bu](x)\ge 0$ for all $x\in\bar D$.

To learn the unknown function, 
we approximate this problem using independent and identically distributed samples $\xi^i=(x_i,y_i)$, $i=1,\ldots,N$,
by
replacing the expectation with a sample average and enforcing the state
constraint only for these samples.
One of our main results shows that, with probability one (\wpone),
the optimal values and solutions of these sample-based problems converge to
those of the risk-neutral problem as $N\to\infty$.
\end{example}

\begin{example}[{Semilinear PDE-constrained optimization under
uncertainty}]
\label{example:semilinear}
Our second motivating example is a PDE-constrained optimization problem with
an uncertain random diffusion coefficient and pointwise state constraints,
which requires the state to stay below a prescribed bound for almost every realization. 
We introduce and analyze a more general formulation in \cref{sec:applications}.

More formally, let $u\in L^2(D)$ be a control, and let
$\xi\in\Xi\subset\mathbb{R}^m$ parametrize the random diffusion coefficient
$\kappa \colon \Xi \to \mathcal{C}(\bar \domain)$. We take
$B$ to be the embedding operator from $L^2(D)$ to the dual space $H^{-1}(D)$ of $H_0^1(D)$. For each $\xi\in\Xi$, let
$y=y(Bu,\xi)\in H_0^1(D)$ denote the weak solution of the semilinear
elliptic PDE
with right-hand side $Bu$:
\begin{align*}
  -\nabla \cdot ( \kappa(\xi) \nabla y) + y^3 = Bu \quad \text{in~$H^{-1}(D)$}, \quad
  y = 0 \quad \text{on~$\partial D$}.
\end{align*}
Using a tracking-type objective functional, we formulate the risk-neutral problem
\begin{align*}
  \min_{u \in \Uad}\,
    \tfrac{1}{2}\mathbb{E}[\|y-y_d\|_{L^2(\domain)}^2]
    +
    \tfrac{\alpha}{2}
    \|u\|_{L^2(D)}^2
\;\; \text{s.t.} \;\;
y(x,\xi) \le y_{\max} 
\;\;
\text{for a.e.\ }
(x,\xi) \in D \times \Xi,
\end{align*}
where $\alpha$ and $y_{\max}$ are positive numbers, $\Uad$ is a nonempty, closed, bounded, and convex set in $L^2(D)$, 
and $y = y(Bu,\xi)$.

As in the previous example, we  approximate the  expectation and the almost sure constraint 
using independent and identically distributed samples.
Beyond the consistency of sample-based
optimal values and solutions, our framework is designed to establish, \wpone, the
convergence of the corresponding sample-based Karush--Kuhn--Tucker (KKT) points to those of the risk-neutral problem as $N \to \infty$. 
\end{example}%

Both examples share three characteristics motivating our general problem formulation provided below: (i) the decision spaces/control spaces are infinite
dimensional, (ii) the constraints are conic and must hold for (almost)
every realization of the uncertainty, and (iii) the problems exhibit a
compactness feature encoded through an operator $B$, which is the key structural property used in our theoretical analysis.

Building on the preceding examples,
we consider a (possibly nonconvex) stochastic program of the following form:
\begin{equation}
\label{eq:stateconstrainedproblem}
\min_{u \in U} \,\mathbb{E}[\mathscr{J}(Bu,\xi)] +\psi(u)\quad \text{s.t.} \quad \mathscr{G}(Bu,\xi) \in K  \quad\text{$\pP$-a.e.~$\xi\in \Xi$}.
\end{equation}

Besides the examples introduced above, a
 problem of this form has  further applications, such as 
field development \cite{Aliyev2016},
optimization of batch reactors  \cite{Ruppen1995}, and natural gas markets \cite{allevi2018evaluating,grimm2019multilevel}.
Here, the random element $\xi$ is defined on a complete probability space $(\Omega, \mathcal{F}, P)$ with images in $\Xi$, the latter being a complete separable metric space equipped with its Borel $\sigma$-algebra. The law $\pP$ of $\xi$ is assumed to be complete and $U$ is the dual to a real separable Banach space. 
Furthermore, $B$ is a linear, sequentially weakly$^*$-to-strongly  continuous operator. 
The objective is allowed to contain a possibly nonsmooth convex term $\psi$ and the constraint is defined with respect to a closed, convex cone $K$ in a real separable Banach space $R$.

In this paper, we
focus on the consistency (as $N\rightarrow \infty$) of optimality conditions for the sample average approximation (SAA) problem \cite{Kleywegt2002,Shapiro2021},
\begin{equation}
\label{eq:stateconstrainedproblem-SAA}
\min_{u \in U} \,\frac{1}{N} \sum_{i=1}^N \mathscr{J}(Bu,\xi^i) +\psi(u)\quad \text{s.t.} \quad \mathscr{G}(Bu,\xi^i) \in K, \quad i=1, \dots, N.
\end{equation}
Here, $\xi^1$, $\xi^2$, $\dots$, are independent 
$\Xi$-valued random elements defined on a probability space 
such that each $\xi^i$ has distribution $\mathbb{P}$.
Our analysis handles not only the consistency of optimal values and solutions to \eqref{eq:stateconstrainedproblem-SAA}, but also the consistency of  Lagrange multipliers appearing in the optimality conditions. The latter property is not only relevant for the application of multiplier-based optimization methods; Lagrange multipliers also provide useful information about the sensitivity of the value function with respect to perturbations of the constraint.

\subsection*{Literature review and discussion}
The SAA approach and related approximation schemes have been widely studied in 
 stochastic programming with constraints. 
For instance, \cite{Lepp1990} established consistency 
of solutions in two-stage convex settings, while 
\cite{Ralph2011} showed convergence of KKT points to points satisfying relaxed 
stationarity conditions. The consistency of optimal solutions in the case of convex programs with a linear objective was shown in \cite{Ramponi2018}. 
Developments for infinite-dimensional stochastic programs have been more recent. 
Solution consistency for the scenario approximation using more general discretization schemes and for infinite-dimensional problems was analyzed in  \cite{PerezAros2021}. 
Consistency analysis of the SAA approach in infinite-dimensional spaces is complicated by the
fact that SAA solutions may not be contained in a compact set
that is independent of $N$, especially since closed unit norm balls in infinite dimensions
are noncompact.
However, large classes of infinite-dimensional problems either allow the construction
of such compact sets
or involve a weakly-to-strongly continuous operator $B$, as demonstrated in 
\cite{Milz2022a,Milz2022d,MilzAugust2024,Milz2023a,Milz2024}.
The sequential weak$^*$-to-strong continuity of the operator $B$ permits us to invoke epigraphical laws of large numbers 
\cite{Hess1996}, extending the consistency analyses in \cite{Milz2022a,MilzAugust2024}
to conic-constrained problems. These ideas were developed in \cite{Milz2022a}.
Consistency analysis for general nonconvex problems appears to be difficult, as 
consistency under epiconvergent objective function approximations 
requires some degree of compactness; see, for example,
\cite[Theorem~2.11]{Attouch1984}.
For convex problems on separable, reflexive Banach spaces, one can involve Mosco epigraphical laws of large numbers instead
\cite{Attouch1990}.
While the sample-based approximation in \eqref{eq:stateconstrainedproblem-SAA}
is commonly referred to as the SAA approach in stochastic programming, the
statistical literature typically uses the term M-estimation
\cite{Gine2016,Geer2010}.%

The classical framework for establishing optimality conditions for a problem of the form \eqref{eq:stateconstrainedproblem} involves the use of the function space $L_{\pP}^\infty(\Xi;R)$ for the image space of the constraint function so that existence of Lagrange multipliers can be shown; this technique goes back to \cite{Rockafellar1976c,Rockafellar1975,Rockafellar1976,Rockafellar1976b} for convex problems with $R$ being an Euclidean space. Recently, these conditions were developed for convex stochastic optimization problems with infinite-dimensional $R$ in \cite{Geiersbach2022,Geiersbach2021}. A main difficulty in this framework is that Lagrange multipliers generally belong to the dual space $L_{\pP}^\infty(\Xi;R)^*$. While in some applications, it is possible to show the existence of Lagrange multipliers in $L_{\pP}^1(\Xi;R^*)$ (for instance, under the assumption of \textit{relatively complete recourse} for $R$ that is also reflexive; 
see \cite{Geiersbach2021}), in many applications, singular Lagrange multipliers appear in the resulting optimality conditions.

The appearance of singularities in optimality conditions is a well-known phenomenon  in deterministic optimal control problems.
Here, a typical choice for $R$ is $\mathcal{C}(O)$ for a compact set $O\subset \R^p$. This function space is compatible with constraint qualifications requiring the existence of interior points (cf.~\cite{Zowe1979}), but singularities in the form of measures appear in the associated Lagrange multipliers. This was first observed for pointwise state constraints in states generated by an elliptic PDE in \cite{Casas1986}. Sometimes, it is possible to improve on the regularity of Lagrange multipliers in deterministic optimal control problems~\cite{Rosch2007,Schiela2009}, but difficulties remain for the numerical computation of solutions. A popular approach is to either penalize the constraint in the objective or smooth it so that the approximate multipliers appearing in the resulting optimality conditions are more regular; see, e.g., \cite{Hintermueller2003,Ito2003,Ulbrich2011}, for early developments involving the use of Moreau--Yosida regularization in optimal control.

In many applications, the constraint function $\mathscr{G}$ from problem \eqref{eq:stateconstrainedproblem} is continuous with respect to its second argument, so that the function space $\mathcal{C}(\Xi;R)$ may be used in place of $L_{\pP}^\infty(\Xi;R)$. This is the setting found in robust and semiinfinite optimization and appears to provide a theoretical advantage in the sense that, provided $R$ is separable,  one can work with sequential arguments for Lagrange multipliers as opposed to nets (see \cite{Geiersbach2022} for the latter). Certainly, this is desirable for the SAA problem \eqref{eq:stateconstrainedproblem-SAA} in the study of the limiting case $N \rightarrow \infty$. This type of functional-analytic framework has been recently applied to optimal control problems under uncertainty; see \cite{Gahururu2022,Geiersbach2024}. Another advantage is that the probability distribution no longer plays a role in the constraint, so that sampling can in principle be done uniformly over $\Xi$, which can improve computational performance considerably; see \cite[Section 5.3]{Geiersbach2025}.

Our problem formulation is related to the PDE-based statistical inverse
problems studied in \cite{Nickl2020,Siebel2025}, but those works do not
impose almost sure conic constraints.
In that setting, $u$ is an unknown parameter,
and the goal is to recover the true parameter from samples.
For consistency analysis and convergence rates,  \cite{Siebel2025} uses data-dependent regularization terms with weights
vanishing as the sample size $N \to \infty$. Adopting this viewpoint here would
amount to augmenting the objective of the sample-based approximation \eqref{eq:stateconstrainedproblem-SAA} by an additional
term $\alpha_N \phi(u)$, where $\alpha_N$
are positive weights with $\alpha_N \to 0$ as $N \to \infty$, and $\phi$
promotes regularity.
While the function $\psi$ in
\eqref{eq:stateconstrainedproblem} and \eqref{eq:stateconstrainedproblem-SAA}
can be used to promote solution regularity, our theory does not cover
additional $N$-dependent regularization terms that vanish as $N\to\infty$.%

\subsection*{Contributions}
The present manuscript provides consistency results that are missing in the literature for infinite-dimensional, generally nonconvex stochastic programs. The main contributions of this work are the following:
\begin{enumerate}[nosep]
    \item We prove the consistency of SAA optimal values and solutions, and show that Moreau--Yosida regularization also yields consistent approximations.
    \item We establish consistency of SAA KKT points under mild smoothness and constraint qualification assumptions. Moreover, we show that the regularization approach yields consistent multipliers.
    \item We apply our theoretical findings to a range of problems, including learning
    and PDE-constrained optimization under uncertainty, and demonstrate how compactness and regularity properties facilitate the analysis.
\end{enumerate}

\subsection*{Outline}
We will introduce our setting in \cref{sec:setup}. We will focus on problems \eqref{eq:stateconstrainedproblem} that are generally nonconvex. To this end, we present a framework that provides compactness with respect to the optimization variable. In \cref{sec:consistency-saa-solutions}, we study the consistency of optimal values and solutions of the SAA problem \eqref{eq:stateconstrainedproblem-SAA} with those of the original problem \eqref{eq:stateconstrainedproblem} as $N \rightarrow \infty$. We use similar arguments in \cref{sec:MY-solutions-consistency} to prove consistency if the SAA problems \eqref{eq:stateconstrainedproblem-SAA} are solved using Moreau--Yosida-type regularization. In \cref{sec:KKT-conditions}, we proceed to KKT conditions. We strengthen our assumptions to allow us to work with constraint functions that are continuous with respect to the parameter space; moreover, we require additional smoothness so that Robinson's constraint qualification can be applied. The resulting optimality conditions for the original problem and the SAA one are shown in \cref{sec:KKT-conditions-original,sec:KKT-conditions-SAA}, respectively. Then, we show consistency of the SAA KKT conditions in \cref{sec:consistency-KKT} and those of the regularized version in \cref{sec:SAA-KKT-MY}. 
In~\cref{sec:sample-complexity-MY}, we turn to sample complexity of the regularized problems. In particular, we derive a rule for the updated penalty parameter as a function of  $N$.
 Then, in~\cref{sec:approximate-feasibility-SAA}, we consider the feasibility of SAA solutions as a function of $N$. Here, we provide an estimate of the probability that an SAA solution is close to the feasible set.
In \cref{sec:applications}, we develop several representative examples from learning and optimal control. We will verify our assumptions with a particular focus on the induced compactness that is characteristic in this setting.

\section*{Notation and conventions} 
The dual pairing will be denoted by $\langle \cdot,\cdot\rangle \coloneqq \langle \cdot,\cdot \rangle_{X^*,X}$, where the Banach space $X = (X, \|\cdot\|_X)$ will be clear from the context.
Strong, weak, and weak$^*$ convergence is denoted by $\rightarrow$, $\rightharpoonup$, and $\rightharpoonup^*$, respectively. The Fr\'echet derivative of $f\colon X \rightarrow \R$ is denoted by $Df \colon X \rightarrow X^*$. The subdifferential (in the sense of convex analysis) of a proper function $\psi$ is denoted by $\partial \psi$.
The closure of a set $C$ is denoted by $\bar{C}$, and with respect to $C$, the indicator function is denoted by $I_C$ 
and the normal cone in a point $x$ is written as $N_C(x)$. If $C\subset X$ is a convex cone, then $C^- = \{x^* \in X^* \mid \langle x^*,x\rangle \leq 0 \, \text{ for all } x\in C\}$ denotes its polar cone.
We denote by $\mathbb{B}_R(0;\rho)$  the open ball
in $R$ with radius $\rho > 0$.

\paragraph{Epiconvergence}
Let $(f_k)$ be a sequence of functions defined on a metric space $X$.
The sequence $(f_k)$ is said to epiconverge towards
a function $f$ on $X$
as $k \to \infty$ if for all $x \in X$ and each sequence $x_k \to x$,
$\liminf_{k \to \infty}\, f_k(x_k) \geq f(x)$,
and for all $x \in X$ there exists a sequence $x_k \to x$ such that
$\limsup_{k \to \infty}\, f_k(x_k) \leq f(x)$.

\paragraph{Probability} We equip metric spaces with their Borel  $\sigma$-field.
Let $(\Theta, \Sigma, \mu)$ be a probability space, and let $X$ be a metric space.
A function $f \colon X \times \Theta \to \mathbb{R} \cup \{\infty\}$ is called random lower semicontinuous if
$\theta \mapsto \{ (x, t) \in X \times \mathbb{R} \colon f(x,\theta) \leq t\}$
is a random closed set \cite[p.\ 429]{Korf2001}.

\paragraph{Function spaces}
If $X$ and $Y$ are Banach spaces, 
we equip their
Cartesian product $X \times Y$
with $\|(x,y)\|_{X \times Y}
\coloneqq (\|x\|_X^2 + \|y\|_Y^2)^{1/2}$
if both $X$ and $Y$ are Hilbert
spaces, and
with $\|(x,y)\|_{X \times Y}
\coloneqq \|x\|_X + \|y\|_Y$
otherwise.
The set $\mathscr{L}(X,Y)$ denotes the space of all bounded linear operators
between the Banach spaces $X$ and $Y$. Given a metric space $K$ and a Banach space $X$, the set $\mathcal{C}(K;X)$ is the set of continuous functions from $K$ into $X$, with the convention $\mathcal{C}(K):=\mathcal{C}(K;\R).$
Given a Banach space $X$ and a measure $\mu$ on a measure space $(\Theta, \Sigma)$, the Bochner space $L_{\mu}^r(\Theta, X) \coloneqq L^r(\Theta,\Sigma, \mu; X)$ is the set of all (equivalence classes of) strongly $\mu$-measurable mappings $y\colon\Theta \rightarrow X$ having finite norm,
with the convention 
$L_{\mu}^r(\Theta) \coloneqq
L_{\mu}^r(\Theta, \mathbb{R})$.
Also, if $\mu$ is the Lebesgue measure on a finite dimensional $\Theta$, we omit writing $\mu$.
Given a positive integer $s$, $1 \leq p\leq \infty$, and a domain $D \subset \R^d$, the Sobolev space $W^{s,p}(D)$ is the set of $L^p(D)$ functions with weak derivatives up to order $s$ that belong to $L^p(D)$. For the case $p=2$, we define $H^s(D)\coloneqq W^{s,2}(D)$. The closure of $\mathcal{C}_0^\infty(D)$ in $W^{s,p}(D)$ is denoted by $W_0^{s,p}(D)$. For $s=1$, its dual space is defined by $W^{-1,p}(D)\coloneqq (W_0^{1,p'}(D))^*$ with $1/p+1/p'=1$.  
We denote by
$\mathrm{BV}(0,T)$ the set of all functions of bounded variation
on $(0,T)$.
Let $M$ be a compact metric space
with metric $d_M$.
Let $\mathrm{Lip}(M)$ be the space of all
real-valued Lipschitz continuous functions  on $M$
equipped with 
\begin{align*}
    \|f\|_{\mathrm{Lip}(M)}
    \coloneqq 
    \max\bigg\{
    \sup_{x \in M}\, |f(x)|, 
     \sup_{x, y \in M, y \neq x}\,
     \frac{|f(y) - f(x)|}{d_M(y,x)}
    \bigg\}.
\end{align*}
If $M$ is pointed
with base point $e_M$, 
$\mathrm{Lip}_0(M)$ 
is the subset of $\mathrm{Lip}(M)$
of functions $u$ with $u(e_M) = 0$.
For a   real Hilbert space $H$, 
we denote by 
$\mathrm{HS}(H)$  the space of Hilbert--Schmidt operators on $H$
equipped with  its canonical inner product.

\section{Stochastic optimization with conical constraints and their sample-based approximations}
\label{sec:setup}
In this section, the framework that will serve as the basis for our investigations is presented. 
We consider the optimization problem \eqref{eq:stateconstrainedproblem} written in the form
\begin{align}
\tag{P}
    \label{eq:reduced-true}
    \min_{u \in U} \,
    F(u)   + \psi(u)
    \quad \text{s.t.} \quad
    G(u,\xi) \in K \quad  \mathbb{P}\text{-a.e. } \xi \in \Xi,
\end{align}
where $F \colon U \to [0,\infty)$ is defined by
\begin{equation*}
    F(u) \coloneqq \E[J(u,\xi)],
\end{equation*}
and $J \colon U \times \Xi \to [0,\infty]$
and $G \colon U \times \Xi \to R$ are defined by
    \begin{equation}
    \label{eq:structure-of-J}
        J(u,\xi) \coloneqq \mathscr{J}(Bu,\xi)
        \quad \text{and} \quad
        G(u,\xi) \coloneqq
        \mathscr{G}(Bu,\xi).
    \end{equation}
Analogously, for a fixed  sample size $N \in \mathbb{N}$, we write the SAA problem  \eqref{eq:stateconstrainedproblem-SAA} as 
\begin{align}
\tag{P\textsubscript{SAA}}
    \label{eq:reduced-saa}
    \min_{u \in U} \,
    \widehat{F}_N(u, \omega)  +\psi(u)
    \quad \text{s.t.} \quad
    G(u,\xi^i(\omega))  \in K,
    \quad  i = 1, \ldots, N,
\end{align}
where $\widehat{F}_N \colon U \times \Omega \to [0,\infty)$ is defined by
\begin{align*}
    \widehat{F}_N(u,\omega) \coloneqq \frac{1}{N}\sum_{i=1}^N J(u,\xi^i(\omega)).
\end{align*}
The second argument of $\widehat{F}_N$ will frequently be omitted. Our assumptions formulated below ensure problem~\eqref{eq:reduced-saa} is well-defined
for each $\omega \in \Omega$. 

We will demonstrate consistency using the following conditions.
\begin{assumption}[{function spaces, objective function, and constraints}]
    \label{assumption:general-assumptions}
\leavevmode\vspace{1pt}
    \begin{enumerate}[nosep]
    \item \emph{Function spaces:} The spaces  $W$,  $R$, $R_0$ are real separable  Banach spaces
    with $R \subset R_0$.
    Moreover, $U$ is the dual to a real separable Banach space.
    \label{itm:spaces}
     \item \emph{Linear operator:} The mapping $B \colon U \to W$ is linear and
     sequentially weakly$^*$-to-strongly 
     continuous.
     \label{itm:B-compact}
    \item  \label{itm:J}
    \emph{Objective:} The map $\mathscr{J} \colon W \times \Xi \to [0,\infty)$
    is  random lower semicontinuous.

    \item \label{itm:G}
    \emph{Constraints:} The map $\mathscr{G} \colon W\times \Xi \rightarrow R_0$
    is Carath\'eodory,
    $\mathscr{G}(Bu,\xi) \in R$
    for all $(u,\xi) \in U\times \Xi$, and
    $\mathscr{G} \colon B(\mathrm{dom}(\psi)) \times \Xi \to R$
    is Carath\'eodory. Moreover,
    $K \subset R$ is a nonempty, closed, and convex cone.

    \item \label{itm:psi-assumptions}
    \emph{Control regularizer:}
    The function $\psi\colon U \rightarrow [0,\infty]$ is proper, convex, and sequentially weakly$^*$  lower semicontinuous with a bounded,
     weakly$^*$ sequentially closed domain.
    \item \label{itm:feasiblepointobjectiveintegrable}
    \emph{Feasibility and finite objective value:}
    Problem \eqref{eq:reduced-true}
    has a feasible point
    with finite objective function value.
    \end{enumerate}
\end{assumption}

A few  remarks about \Cref{assumption:general-assumptions}  are in order.
In the case where $U$ is reflexive and separable, it is sufficient to require weak-to-strong continuity of $B$ and  lower semicontinuity of $\psi$. An important example of $U$ is the space of bounded variation functions  (see \cite[Remark~3.12]{Ambrosio2000}). A further example of $U$ is the space of all finite signed regular Borel measures on a compact metric space.
Our assumption on \( U \) ensures 
bounded sequences in \( U \) have weak\(^*\) convergent subsequences: the closed
unit ball in $U$ is sequentially compact in the weak\(^*\) topology
\cite[Theorem 8.5]{Alt2016}.
Although separability of the predual is unrelated to SAA measurability issues, a nonseparable space \( U \) may prevent us from showing
measurability of SAA solutions. However, our consistency analysis
shows that such measurability is not required. In contrast,
separability of \( W \) is crucial for applying the epigraphical law of large
numbers \cite{Hess1996}.

\Cref{assumption:general-assumptions}\ref{itm:spaces},\ref{itm:psi-assumptions} ensures 
\(\mathrm{dom}(\psi)\) is weak\(^*\) sequentially compact. Indeed, any 
\((u_N) \subset \mathrm{dom}(\psi)\) is bounded, and since \( U \) is the dual
of a separable Banach space, it admits a weak\(^*\) convergent subsequence with weak\(^*\) limit in
\(\mathrm{dom}(\psi)\), because it is weak\(^*\) 
sequentially closed. Boundedness of \(\mathrm{dom}(\psi)\) is a rather mild
assumption; for problems with unbounded domains, coercivity of $\psi$ may be used to 
augment \(\mathrm{dom}(\psi)\) on a bounded level set.

Our assumption on $B$ is exploited in the consistency analysis. 
Typical instances of \(B\) arise from suitable compact embeddings. For example, if \(U\) is
reflexive, \(B\) may be taken as the adjoint of a suitable compact Sobolev embedding, such as
the adjoint operator of $H^1_0(D) \hookrightarrow L^2(D)$ for a bounded Lipschitz domain $D \subset \mathbb{R}^d$.
Further examples include compact embeddings from \(\mathrm{BV}(0,T)\) into
separable Lebesgue spaces $L^p(0,T)$,
$1 \leq p < \infty$,
on bounded intervals $[0,T]$, as well as the canonical
embedding \(\mathrm{Lip}(M)\hookrightarrow \mathcal{C}(M)\) on a compact metric
space \(M\).

Importantly, the introduction of $B$ into our framework equips the objective and the constraint function with a compactness property that is necessary to apply the epigraphical law of large numbers. The map $\mathscr{J}$ on its own is only random lower semicontinuous, but the objective $J(u,\xi) = \mathscr{J}(Bu,\xi)$ is sequentially weakly$^*$  lower semicontinuous with respect to $u$. Likewise,
\Cref{assumption:general-assumptions}\ref{itm:G}  ensures  \(G \) given by $G(u,\xi) = \mathscr{G}(Bu,\xi)$ is well-defined, and that for each \( \xi \in \Xi \), any sequence \( (u_N) \subset \mathrm{dom}(\psi) \) with \( u_N \rightharpoonup^* u \)  satisfies \( G(u_N, \xi) \to G(u, \xi) \) in \( R \). Notable in our construction of $\mathscr{G}$ is the possible restriction $R \subset R_0$ in the image space appearing in the presence of the operator $B$. In \Cref{example:semilinear}, which is analyzed further in \cref{subsect:semilinear-elliptic-pde}, this is used to delicately balance continuity and compactness in the operator mapping the control $u$ to the solution $y$ of the PDE.

\Cref{assumption:general-assumptions}  ensures the existence of solutions.

\begin{lemma}
    If \Cref{assumption:general-assumptions} holds, then 
    \emph{(a)} problem \eqref{eq:reduced-true} has solutions
    and \emph{(b)} \wpone, the SAA problems \eqref{eq:reduced-saa} 
    have solutions.
\end{lemma}

\begin{proof}
    (a)
    Let $(u_N)$ be a minimizing sequence for
    \eqref{eq:reduced-true}. 
    In particular, $u_N$ is feasible and $u_N \rightharpoonup^* u \in \mathrm{dom}(\psi)$
    (along a subsequence).
    For each $N \in \mathbb{N}$, 
    there exists a subset $\Xi_N \subset \Xi$
    such that $\Xi_N$ occurs \wpone~and $G(u_N, \xi) \in K$ for all $\xi \in \Xi_N$.
    For each $\xi \in \cap_{N \in \mathbb{N}}\, \Xi_N$, 
    $K \ni  G(u_N,\xi) \to G(u,\xi) $ in $R$.
    Hence, $u$ is feasible for \eqref{eq:reduced-true}.
    Since  $F(\cdot) = \mathbb{E}[\mathscr{J}(B\cdot,\xi)]$ 
    is
    sequentially weakly$^*$  lower semicontinuous, 
    $u$ solves \eqref{eq:reduced-true}.

    (b) Let $\bar u$ be feasible for \eqref{eq:reduced-true}, 
    and let $N \in \mathbb{N}$.
    Then \wpone, $\bar u$ is feasible for \eqref{eq:reduced-saa}.
    The existence proof is now similar to part (a).
\end{proof}

Finally, we discuss properties of $B$.

\begin{lemma}
\label{lem:properties-B}
Under
\Cref{assumption:general-assumptions}\ref{itm:spaces},\ref{itm:B-compact},
the operator $B$ is continuous and  $B(\mathrm{dom}(\psi))$
is compact.
\end{lemma}
\begin{proof}
Since $u_k \to u$ implies $u_k \rightharpoonup^* u$, we 
obtain the first assertion.
Let $(w_N) \subset B(\mathrm{dom}(\psi))$.
Then, for each $N\in \mathbb{N}$, we have
$w_N = B u_N$ for some $u_N \in \mathrm{dom}(\psi)$.
Hence $u_N \rightharpoonup^* u \in \mathrm{dom}(\psi)$
(at least along a subsequence), 
yielding $w_N \to Bu$.
\end{proof}

\section{Consistency of optimal values and solutions}
\label{sec:optimal-values-solutions}
In this section, we analyze the consistency of optimal values and solutions to those of     problem \eqref{eq:reduced-true}, both for the SAA problem and for a variant involving regularization of the constraint.

\subsection{Consistency of SAA optimal values and solutions}
\label{sec:consistency-saa-solutions}

We proceed with an analysis of solutions to the SAA problem. Let $\vartheta^\star$ and $\widehat{\vartheta}_N^\star$  be the optimal values of \eqref{eq:reduced-true} and~\eqref{eq:reduced-saa}, respectively, and let  $u_N^*$ be
a solution to the SAA problem.
The main result of this section is the following.

\begin{proposition}
    \label{prop:consistencyreducedsaa}
    Suppose that \Cref{assumption:general-assumptions} holds.
    Then
    \wpone, $\widehat{\vartheta}_N^\star \to \vartheta^\star$ as $N \to \infty$
    and each weak$^*$ limit point of $(u_N^*)$
    solves problem \eqref{eq:reduced-true}.
\end{proposition}

Let us clarify the assertion of \Cref{prop:consistencyreducedsaa}.
It asserts that for $P$-a.e.\
$\omega \in \Omega$, 
$\widehat{\vartheta}_N^\star(\omega) \to \vartheta^\star$ as $N \to \infty$
and each weak$^*$ limit point of $(u_N^*(\omega))$
solves \eqref{eq:reduced-true}.
This does not require that $ u_N^* $ be (strongly) measurable.

We prepare for the proof of \Cref{prop:consistencyreducedsaa}.
We define the multifunction $\mathcal{U} \colon \Xi  \rightrightarrows U$
by $\mathcal{U}(\xi) \coloneqq \{ u \in U \mid G(u,\xi) \in K \}$.

\begin{lemma}
\label{lem:objectivesepiconvergent}
    If \Cref{assumption:general-assumptions} holds, then 
    \begin{enumerate}[nosep]
        \item      \label{itm:lem:objectivesepiconvergent:feasibility}
        for $u \in \mathrm{dom}(\psi)$, $\E[I_{\mathcal{U}(\xi)}(u)] = 0$
        if and only if $G(u,\xi)\in K$
        for $\mathbb{P}$-a.e.\ $\xi \in \Xi$,
        \item
        \label{eq:epiconvergenceF+I}
        \wpone, for all 
         $u \in \mathrm{dom}(\psi)$
        and all sequences
        $(u_N) \subset \mathrm{dom}(\psi)$
        such that 
        $u_N \rightharpoonup^* u$,
        $\liminf_{N \to \infty}\, \widehat{F}_N(u_N)
        \geq F(u)$
        and $\liminf_{N \to \infty}\, (1/N)\sum_{i=1}^N I_{\mathcal{U}(\xi^i)}(u_N)
        \geq \E[I_{\mathcal{U}(\xi)}(u)]$, and
  \item
\label{eq:epiconvergenceFB+I}
if $u$ is a feasible point for
\eqref{eq:reduced-true},
then we have \wpone, 
       $\widehat{F}_N(u) \to F(u)$
       and $(1/N)\sum_{i=1}^N I_{\mathcal{U}(\xi^i)}(u) \to  \E[I_{\mathcal{U}(\xi)}(u)]$
as $N \to \infty$.
    \end{enumerate}
\end{lemma}

\begin{proof}
    \begin{enumerate}[nosep,wide]
        \item
        For  $(u,\xi) \in \mathrm{dom}(\psi) \times \Xi$,
        $I_{\mathcal{U}(\xi)}(u) = I_K(G(u,\xi))$.
        Since $U$ may be nonseparable, we  only show that
        for each fixed $u \in \mathrm{dom}(\psi)$, 
        $\xi \mapsto I_{\mathcal{U}(\xi)}(u)$ is measurable.
        For $t \geq 0$, we have
        $\{\xi \in \Xi \mid I_K(G(u,\xi)) \leq t\}
        = \{\xi \in \Xi \mid G(u,\xi) \in K\}$.
        Since $G(u,\cdot)$ is measurable,
        so is $\{\xi \in \Xi \mid I_K(G(u,\xi)) \leq t\}$.
        Hence  $\xi \mapsto I_{\mathcal{U}(\xi)}(u)$ is measurable.
        The equivalence now
        follows from the definition of the indicator function.

    \item Since $\mathscr{J}$ is random lower semicontinuous by \Cref{assumption:general-assumptions}\ref{itm:J},
    and $W$ is  separable, we can apply \cite[Theorem~5.1]{Hess1996} to obtain the epiconvergence of
$w \mapsto (1/N)\sum_{i=1}^N \mathscr{J}(w,\xi^i)$ to $w \mapsto \E[\mathscr{J}(w, \xi)]$ \wpone.
Combined with the fact that $B$
is sequentially weakly$^*$-to-strongly continuous
by  \Cref{assumption:general-assumptions}\ref{itm:B-compact}, we obtain the first assertion.

For each $\xi \in \Xi$, the epigraph of
$B(\mathrm{dom}(\psi)) \ni w  \mapsto I_{K}(\mathscr{G}(w,\xi))$ is given by
$\{(w, t) \in B(\mathrm{dom}(\psi)) \times \mathbb{R} \mid 
\mathscr{G}(w,\xi) \in K, t \geq 0\}$,
which is an inverse image.
Combined with \cite[Theorem~8.29]{Aubin2009}, 
  $B(\mathrm{dom}(\psi)) \times \Xi \ni  (w, \xi) \mapsto I_{K}(\mathscr{G}(w,\xi))$ is random lower semicontinuous.
Hence, we can apply \cite[Theorem~5.1]{Hess1996} to obtain the almost sure epiconvergence of
$w \mapsto (1/N)\sum_{i=1}^N I_{K}(\mathscr{G}(w,\xi^i))$ to $w \mapsto \E[ I_{K}(\mathscr{G}(w,\xi))]$
on $B(\mathrm{dom}(\psi))$.
Using the sequential weak$^*$-to-strong continuity of $B$, we have proven the second assertion.

\item The strong law of large numbers (SLLN) implies the assertion.
\end{enumerate}
\end{proof}

\begin{proof}[{Proof of \Cref{prop:consistencyreducedsaa}}]
    This proof uses techniques developed in \cite{Milz2022a}.
    Let $\bar u$ be a solution to \eqref{eq:reduced-true}.
    According to \Cref{lem:objectivesepiconvergent},
    there exists $\Omega_0 \in \mathcal{F}$
    such that $P(\Omega_0) = 1$
    and for each $\omega \in \Omega_0$,
    for all sequences $u_N \rightharpoonup^* u$, we have
        $\liminf_{N \to \infty}\, \widehat{F}_N(u_N,\omega)
        \geq F(u)$
        and $\liminf_{N \to \infty}\, (1/N)\sum_{i=1}^N I_{\mathcal{U}(\xi^i(\omega)
        )}(u_N)
        \geq \E[I_{\mathcal{U}(\xi)}(u)]$. Additionally, we have 
    $\widehat{F}_N(\bar u,\omega) \to F(\bar u)$
       and $(1/N)\sum_{i=1}^N I_{\mathcal{U}(\xi^i(\omega))}(\bar u) \to  0$
as $N \to \infty$.
    The following statements are valid for each fixed $\omega \in \Omega_0$,
    so we omit writing it.

    Let $(u_N^*)_{\mathcal{N}}$ be a subsequence of $(u_N^*)$
    converging weakly$^*$ to $u^*$ as $\mathcal{N} \ni N \to \infty$.
    We define $\widetilde u_N \coloneqq u_N^*$
    if $N \in \mathcal{N}$ and $\widetilde u_N \coloneqq u^*$
    otherwise. Hence
    $\widetilde u_N \rightharpoonup^*  u^* \in \mathrm{dom}(\psi)$.
    The sequential weak$^*$ lower semicontinuity of $\psi$
    ensures
    \begin{align}
        \nonumber
        F(u^*) + \psi(u^*) + \E[I_{\mathcal{U}(\xi)}(u^*)]
        &\leq
        \liminf_{ N \to \infty} \,
        \widehat{F}_N( \widetilde u_N)
        +\psi(\widetilde u_N) +
        \frac{1}{N}\sum_{i=1}^N I_{\mathcal{U}(\xi^i)}(\widetilde u_N)
        \\
        \label{eq:Dec42023620}
        &\leq \liminf_{\mathcal{N}\ni N \to \infty} \,
        \widehat{F}_N( u_N^*)
        +\psi(u_N^*) +
        \frac{1}{N}\sum_{i=1}^N I_{\mathcal{U}(\xi^i)}(u_N^*)
        \\
        \nonumber 
        & = \liminf_{\mathcal{N}\ni N \to \infty} \, \widehat{\vartheta}_N^\star.
    \end{align}

    We also have
    \begin{equation}
    \label{eq:limsup-inequality-proof}
        \limsup_{N \to \infty} \, \widehat{\vartheta}_N^\star
        \leq
        \lim_{N \to \infty} \, \widehat{F}_N(\bar u)
        + \psi(\bar u)
        + \frac{1}{N}\sum_{i=1}^N I_{\mathcal{U}(\xi^i)}(\bar u)
    = F(\bar u) + \psi(\bar u) + \E[I_{\mathcal{U}(\xi)}(\bar u)].
    \end{equation}
    Combined with \eqref{eq:Dec42023620}
    and $F(\bar u) + \psi(\bar u) < \infty$,
     $\E[I_{\mathcal{U}(\xi)}(u^*)]=0$.
    Hence
    $u^*$ solves \eqref{eq:reduced-true}.
    Consistency of the optimal values can now be shown by combining \eqref{eq:Dec42023620} and   \eqref{eq:limsup-inequality-proof}.
\end{proof}

\subsection{Consistency of SAA solutions via regularization}
\label{sec:MY-solutions-consistency}
Approaches relying on the regularization of the conical constraint in \eqref{eq:reduced-true} are frequently used in the numerical solution of state-constrained problems (these are also called \textit{penalty methods} in the literature). As in the previous section, we will see that optimal values and solutions to the corresponding regularized SAA problem are consistent with that of \eqref{eq:reduced-true}. We will rely on a Moreau--Yosida-type regularization $\beta\colon R \rightarrow [0,\infty)$ with the following property.
\begin{assumption}
    \label{assumption:penalty}
    The function $\beta \in\mathcal{C}( R,[0,\infty))$ satisfies $\beta(r) = 0$ if and only if $r \in K$.
\end{assumption}

Now, we define with $\gamma >0$ the regularized problem by
\begin{equation}
\label{eq:stateconstrainedproblem-robust-MY}
\min_{u \in U} \, F(u) + \psi(u) +\gamma\E[\beta(G(u,\xi))].
\end{equation}
For the approximation of the last term in \eqref{eq:stateconstrainedproblem-robust-MY}, we 
define $\widehat{\varphi}_N  \colon U  \times \Omega \to [0,\infty)$ by
\begin{align*}
        \widehat{\varphi}_N(u,\omega) \coloneqq \frac{1}{N}\sum_{i=1}^N \beta(G(u,\xi^i(\omega))).
\end{align*}
As with $\widehat{F}_N$, we often omit writing the second argument of $\widehat{\varphi}_N$. For fixed $\gamma_N > 0$,
the SAA regularized problem is given by
\begin{align}
    \label{eq:penaltyproblem}
    \min_{u \in U} \,
    \widehat{F}_N(u)  +\psi(u)
    + \gamma_N \widehat{\varphi}_N(u).
\end{align}

Under \Cref{assumption:general-assumptions,assumption:penalty},  
the problem \eqref{eq:penaltyproblem} has a solution $u_{N}^{\gamma_N}$ (with optimal value denoted by $\widehat{\vartheta}_{N}^{\gamma_N}$)
for every $N \in \mathbb{N}$
and $\gamma_N > 0$.

\begin{proposition}
\label{prop:MY-state-robust}
    Let \Cref{assumption:general-assumptions,assumption:penalty} be satisfied.
    Let $(\gamma_N) \subset (0,\infty)$
    satisfy $\gamma_N \to \infty$ as $N \to \infty$.
    Then \wpone, $\widehat{\vartheta}_{N}^{\gamma_N} \to \vartheta^\star$ as $N \to \infty$ and each weak$^*$ limit point of
    $(u_{N}^{\gamma_N})$
    solves \eqref{eq:reduced-true}.
\end{proposition}

\begin{proof}
Let $\varphi \colon \mathrm{dom}(\psi) \to [0,\infty)$ be defined by
$\varphi(u) \coloneqq \E[\beta(G(u,\xi))]$.
Using  \cite[Theorem~5.1]{Hess1996}, we have
\wpone, 
$w \mapsto (1/N) \sum_{i=1}^N \beta(\mathscr{G}(w, \xi^i))$
epiconverges to
$w \mapsto \mathbb{E}[\beta(\mathscr{G}(w, \xi))]$
on $B(\mathrm{dom}(\psi))$.
Combined with the sequential weak$^*$-to-strong continuity of $B$, we have \wpone,
for all $u_N \rightharpoonup^* u$,
$\liminf_{N \to \infty}\, \widehat{\varphi}_N(u_N)
\geq \varphi(u)$.
Let  $\bar u$ be a solution to \eqref{eq:reduced-true}.
\Cref{lem:objectivesepiconvergent}\ref{eq:epiconvergenceFB+I} ensures that \wpone,  $\widehat{F}_N (\bar u)\to F(\bar u)$, $\widehat{\varphi}_N(\bar u)  \to \varphi(\bar u)$, 
and
  $
        (1/N)\sum_{i=1}^N I_{\mathcal{U}(\xi^i)}(\bar u)
    \to \E[I_{\mathcal{U}(\xi)}(\bar u)]
 = 0$
as $N \to \infty$.

Let $(u_{N}^{\gamma_N})_{\mathcal{N}}$ be a subsequence of
    $(u_{N}^{\gamma_N})$ weakly$^*$ converging to $u^*$ as $\mathcal{N} \ni N \to \infty$.
We first show that $u^*$ is feasible for \eqref{eq:reduced-true}.
    Using  $\widehat{F}_N \geq 0$
    and $\psi \geq 0$, as well as the optimality of $u_N^{\gamma_N}$, 
    \begin{align*}
        \gamma_N\widehat{\varphi}_N(u_{N}^{\gamma_N}) \leq \widehat{F}_N(u_{N}^{\gamma_N})
        + \psi( u_{N}^{\gamma_N})+\gamma_N\widehat{\varphi}_N(u_{N}^{\gamma_N})
        \leq \widehat{F}_N (\bar u)
        + \psi(\bar u)+ \gamma_N\widehat{\varphi}_N(\bar u).
    \end{align*}
    Hence 
    \begin{align*}
        \widehat{\varphi}_N(u_{N}^{\gamma_N})
        \leq \frac{1}{\gamma_N} \big(\widehat{F}_N (\bar u)+\psi(\bar u)\big)
        + \widehat{\varphi}_N(\bar u)
        \to 0 \quad \text{as} \quad N \to \infty.
    \end{align*}
    Therefore, 
    $
      \varphi(u^*) \leq \liminf_{\mathcal{N} \ni N \to \infty}\, \widehat{\varphi}_N(u_{N}^{\gamma_N}) = 0
    $,
    implying that $u^*$ is feasible for \eqref{eq:reduced-true}.

    Now, we prove that $u^*$ solves \eqref{eq:reduced-true}.
    Since 
    $\psi$ is sequentially weakly$^*$  lower semicontinuous (see \Cref{assumption:general-assumptions}\ref{itm:psi-assumptions}), 
    and $\widehat{\varphi}_N \geq 0$, we have
    \begin{align}
        \label{eq:Dec220231238}
        \begin{aligned}
        F(u^*) + \psi(u^*)
         &\leq \liminf_{ \mathcal{N} \ni N \to \infty} \, \widehat{F}_N(u_{N}^{\gamma_N}) + \psi(u_{N}^{\gamma_N})
         \\
         &\leq \liminf_{\mathcal{N} \ni N \to \infty} \, \widehat{F}_N(u_{N}^{\gamma_N})
         + \psi(u_{N}^{\gamma_N})+\gamma_N\widehat{\varphi}_N(u_{N}^{\gamma_N}).
        \end{aligned}
    \end{align}
    Let $\bar u$ be a solution to \eqref{eq:reduced-true}.
    Using \eqref{eq:Dec220231238}, $\gamma_N\widehat{\varphi}_N(\bar u) \leq (1/N) \sum_{i=1}^N I_{\mathcal{U}(\xi^i)}(\bar u)$,
    and the optimality of $u_{N}^{\gamma_N}$, we obtain
    \begin{align*}
        F(u^*) +\psi(u^*)
         \leq \limsup_{N \to \infty} \, \widehat{F}_N(\bar u)
        +\psi(\bar u)+ \frac{1}{N} \sum_{i=1}^N I_{\mathcal{U}(\xi^i)}(\bar u) = F(\bar u) +\psi(\bar u).
    \end{align*}
    Hence $u^*$ solves \eqref{eq:reduced-true}.
    Consistency of the optimal value can now be shown using standard arguments.
\end{proof}

\section{Consistency of SAA KKT points}
\label{sec:KKT-conditions}
In this section, we focus on optimality conditions. 
We commence with a formulation of the KKT conditions for problem \eqref{eq:reduced-true} in \cref{sec:KKT-conditions-original}. Then, in \cref{sec:KKT-conditions-SAA}, we consider the KKT conditions for the SAA problem. We then analyze their consistency in \cref{sec:consistency-KKT}. We close the section by analyzing the regularized problem in \cref{sec:SAA-KKT-MY}.

\subsection{KKT conditions}
\label{sec:KKT-conditions-original}
Let us consider the robust counterpart to problem \eqref{eq:reduced-true}:
\begin{equation}
\label{eq:stateconstrainedproblem-robust}
\min_{u \in U} \, F(u) +\psi(u) \quad \text{s.t.} \quad G(u,\xi) \in K  \quad \text{for all} \quad  \xi \in \Xi.
\end{equation}

Problem~\eqref{eq:stateconstrainedproblem-robust} coincides with \eqref{eq:reduced-true} if we assume continuity of $\xi \mapsto G(u,\xi)$ for every $u \in U$ and that $\Xi$ is the support (i.e., the closed set that is the intersection of all closed sets having probability equal to one; cf.~\cite[p.~77]{zbMATH05116074}) of $\pP$. In particular, by following the arguments from \cite[Lemma 9]{Geiersbach2024}, we have
\[G(u,\xi) \in K \quad \pP\text{-a.e.}~\xi \in \Xi \quad \text{if and only if} \quad G(u,\xi) \in K \quad \text{for all} \quad  \xi \in \Xi.\]

We will require the following additional assumptions to show consistency of the KKT system. Recall the definitions given in  \eqref{eq:structure-of-J}.

\begin{assumption}[{Compact sample space; smooth objective and constraint}]
\label{assumption:continuous-setting}
\leavevmode\vspace{1pt}
\begin{enumerate}[nosep]
\item \label{itm:Xi-compact}\emph{Sample space:} The metric space $\Xi$ is compact.
\item \label{itm:j-smooth} 
\emph{Smooth objective:} 
Let \( U_0 \subset U \) and \( W_0 \subset W \) 
be open sets  with \( \mathrm{dom}(\psi) \subset U_0 \) and  
\( B(U_0) \subset W_0 \). For all \( \xi \in \Xi \), 
\( \mathscr{J}(\cdot, \xi) \) is continuously differentiable on $W_0$, and for all  
\( w \in W_0 \),  \( D_w \mathscr{J}(w, \cdot) \) is  
measurable. Moreover, there exists an integrable function  
\( L \colon \Xi \to [0, \infty) \) such that  
\( \|D_w \mathscr{J}(w, \xi)\| \leq L(\xi) \) for all \( (w, \xi) \in W_0 \times \Xi \).

\item \label{itm:G-smooth} \emph{Smooth constraint:}  For each \( \xi \in \Xi \), \( \mathscr{G}(\cdot, \xi) \) is continuously  
differentiable on $W_0$, and for all \( w \in W_0 \), \( D_w \mathscr{G}(w, \cdot) \) is measurable. Moreover,  
\( \mathscr{G} \colon B(U_0) \times \Xi \to R \) and  
\( D_w \mathscr{G} \colon B(U_0) \times \Xi \to \mathscr{L}(W, R) \) are  
continuous.
    \end{enumerate}
\end{assumption}

\begin{remark}
\label{rem:assumptions-KKT}
\Cref{assumption:continuous-setting}\ref{itm:j-smooth} and
\cite[Lemma C.3]{geiersbach_stochastic_2021}
ensure that
$F$ is Fr\'echet differentiable on $U_0$.
By the chain rule,
\[
 \langle D_uJ(u,\xi),h\rangle =  \langle D_w \mathscr{J}(Bu,\xi),Bh\rangle,
\]
so that $(u,h) \mapsto \langle D_uJ(u,\xi),h\rangle$ is sequentially weakly$^*$-to-strongly continuous for every $\xi \in \Xi$ on account of this property of $B$.
Moreover, since $\lVert D_w \mathscr{J}(w,\xi)\rVert \leq L(\xi)$ for integrable $L$, we have by the dominated convergence theorem that given a weakly$^*$ convergent sequence $(u_N,h_N)$,
\begin{align*}
\lim_{N \rightarrow \infty} \langle DF(u_N),h_N\rangle &=\lim_{N\rightarrow \infty} \E[\langle D_w \mathscr{J}(Bu_N,\xi),Bh_N\rangle] \\
&= \E[\lim_{N\rightarrow \infty} \langle D_w \mathscr{J}(Bu_N,\xi),Bh_N\rangle]
\end{align*}
so that $(u,h)\mapsto \langle DF(u),h\rangle$ is sequentially weakly$^*$-to-strongly continuous. Here, we used the fact that $\mathscr{J}$ is continuously differentiable with respect to its first argument.
Similarly, $(u,h) \mapsto  D_u G(u,\xi) h =  D_w \mathscr{G}(Bu,\xi) B h $  is a sequentially weakly$^*$-to-strongly continuous mapping 
for every $\xi \in \Xi$.
\end{remark}

In the following, it will be convenient to work with the mapping
\begin{equation*}
    \mathcal{G} \colon U_0 \to \mathcal{C}(\Xi; R),
    \quad \mathcal{G}(u) \coloneqq G(u,\cdot).
\end{equation*}
Notice that by \Cref{assumption:continuous-setting}\ref{itm:Xi-compact},\ref{itm:G-smooth}, $\mathcal{G}$ is
continuously differentiable
with 
$D\mathcal{G}(u)
= D_u G(u,\cdot)$, 
so that for each $u \in U_0$ the adjoint operator $D\mathcal{G}(u)^*$ is a bounded, linear mapping from $\mathcal{C}(\Xi;R)^*$ to $U^*$. For the following optimality conditions, it will be convenient to define the (nonempty, closed, and convex) cone
\[
\mathcal{K} \coloneqq \{r \in \mathcal{C}(\Xi; R) \mid r(\xi) \in K \, \,  \text{for all} \,  \xi \in \Xi \}.
\]

\begin{lemma}
\label{lemma:KKT-continuous}
Let
\Cref{assumption:general-assumptions,assumption:continuous-setting} be satisfied.
Suppose that $\mathcal{K}$ has a nonempty interior and the constraint qualification
\begin{equation}
\label{eq:CQ-robust-nonempty-interior}
\text{there exists} \quad  \hat{u}\in \dom\psi \quad \text{such that}  \quad \mathcal{G}(u^*) +D\mathcal{G}(u^*)(\hat{u}-u^*) \in \textup{int}(\mathcal{K}).
\end{equation}
is satisfied at a local optimum $u^*$ to problem~\eqref{eq:reduced-true}.
Then, there exists a Lagrange multiplier $\lambda^* \in \mathcal{C}(\Xi;R)^*$ such that
\begin{equation}
\label{eq:KKT-abstract-robust}
0 \in DF(u^*) + \partial \psi(u^*) + D \mathcal{G}(u^*)^*\lambda^*,\quad \lambda^* \in \mathcal{K}^-,
\quad  \langle \lambda^*, \mathcal{G}(u^*)\rangle = 0.
\end{equation}
\end{lemma}

\begin{proof}
Let $t^* \coloneqq \psi(u^*)$; then, the point $(u^*,t^*)$ is a local solution to
\[
\min_{(u,t) \in \mathrm{epi}(\psi)} \quad F(u) + t \quad \text{s.t.} \quad
\mathcal{G}(u) \in \mathcal{K}.
\]
The constraint qualification \eqref{eq:CQ-robust-nonempty-interior} 
implies 
with $\hat{t} \coloneqq \psi(\hat u)$, 
$\hat x \coloneqq (\hat u, \hat t)$, 
$x^* \coloneqq (u^*, t^*)$,
and $\mathcal{H}(u, t) \coloneqq \mathcal{G}(u)$,
\begin{align*}
    \mathcal{H}(x^*)
    +
    D\mathcal{H}(x^*)
    (\hat x - x^*)
    \in \mathrm{int}(\mathcal{K}).
\end{align*}
Let $\hat y$ be the vector on the left-hand side.
Hence, there exists $\varepsilon > 0$ such that
$\hat y + \mathbb{B}_{\mathcal{C}(\Xi; R)}(0;\varepsilon) \subset \mathcal{K}$.
This yields $\mathbb{B}_{\mathcal{C}(\Xi; R)}(0;\varepsilon) \subset \hat y - \mathcal{K}$, 
which in turn implies the following constraint qualification 
(see \cite[eq.\ (3.13)]{Bonnans2013}) for the reformulated problem:
\begin{align*}
    0 \in 
    \mathrm{int}
    \{
    \mathcal{H}(x^*)
    +
    D\mathcal{H}(x^*)
    (
    \mathrm{epi}(\psi)-
    x^*) 
    - \mathcal{K}
    \}.
\end{align*}
Hence, its  set of Lagrange multipliers at $(u^*,t^*)$ is nonempty (see \cite[Theorem 3.9]{Bonnans2013}).
Defining $L(u,t,\lambda) \coloneqq F(u)+t+\langle \lambda , \mathcal{G}(u)\rangle$,
there exists $\lambda^* \in \mathcal{C}(\Xi;R)^*$ such that
(cf.\ \cite[eq.\ (3.17)]{Bonnans2013})
\begin{equation*}
-D_{(u,t)} L(u^*,t^*, \lambda^*)  \in N_{\mathrm{epi}(\psi)}(u^*,t^*),
\quad \lambda^* \in \mathcal{K}^-,
\quad
\text{and} \quad \langle \lambda^*, \mathcal{G}(u^*)\rangle = 0.
\end{equation*}
Combined with the representation
of $N_{\mathrm{epi}(\psi)}(u^*,t^*)$ provided by
 \cite[Remark~2.117]{Bonnans2013}, we obtain \eqref{eq:KKT-abstract-robust}.
\end{proof}

The conditions in \eqref{eq:KKT-abstract-robust} together with the feasibility condition $\mathcal{G}(u^*) \in \mathcal{K}$ make up the KKT conditions for problem \eqref{eq:reduced-true}. Generally speaking, they are only necessary conditions for optimality. In the case where \eqref{eq:reduced-true} is a convex program, then these become sufficient for optimality, as shown in the following result.

\begin{corollary}
\label{cor:necessary-and-sufficient}
In addition to the assumptions in \Cref{lemma:KKT-continuous}, suppose that $F$ is convex and $G$ is convex in its first argument with respect to $(-K)$, 
meaning that for each $u_1, u_2 \in U$, $t \in [0,1]$,
\[
t G(u_1,\xi) + (1-t)G(u_2,\xi) - G(tu_1 +(1-t) u_2,\xi) \in -K \quad \text{for all} \quad \xi \in \Xi.
\]
Then, the conditions  in \eqref{eq:KKT-abstract-robust} are sufficient conditions for optimality of a feasible $u^*$.
\end{corollary}
\begin{proof}
Sufficiency of the conditions holds when problem \eqref{eq:stateconstrainedproblem-robust} is convex by \cite[Proposition 3.3]{Bonnans2013}.
\end{proof}

\subsection{KKT conditions for the SAA problem}
\label{sec:KKT-conditions-SAA}

Now, we demonstrate that the KKT conditions of \eqref{eq:reduced-saa}
provide approximate KKT conditions for \eqref{eq:reduced-true}.

The KKT conditions of \eqref{eq:reduced-saa} are given by
\begin{align}
\label{eq:KKTcontinuousmodel}
\begin{aligned}
    &0  \in
  D\widehat{F}_N(u_N^*) +\partial \psi(u_N^*)+
    \frac{1}{N}\sum_{i=1}^N D_u G(u_N^*, \xi^i)^* \mu_{N,i}^*,
    \\
    & \mu_{N,i}^* \in K^{-},
    \quad \langle \mu_{N,i}^*, G(u_N^*,\xi^i) \rangle = 0,
    \quad
    G(u_N^*,\xi^i) \in K,
    \quad i = 1, \ldots, N.
\end{aligned}
\end{align}

For the consistency results in   \cref{sec:consistency-KKT}, since we do not rely on the convexity of \eqref{eq:reduced-saa}, we will only suppose that there exists a sequence of KKT points $(u_N^*, \lambda_N^*)$ satisfying \eqref{eq:KKTcontinuousmodel}. Of course, this sequence exists, for instance, provided $u_N^*$ is a local solution to \eqref{eq:reduced-saa} that satisfies a constraint qualification.

We define the evaluation functional $\delta_\xi \colon \mathcal{C}(\Xi; R) \to R$ by $\delta_\xi(v) \coloneqq  v(\xi)$.
Given the Lagrange multipliers $\mu_{N,i}^*$ in \eqref{eq:KKTcontinuousmodel} and $v \in \mathcal{C}(\Xi;R)$, we define the approximate multiplier
$\lambda_N^*$ for \eqref{eq:reduced-true} by
\begin{equation}
\label{eq:identification-lambda-mu}
    \langle \lambda_N^*, v \rangle  \coloneqq \frac 1 N \sum_{i=1}^N \langle \mu_{N,i}^*, \delta_{\xi^i}(v) \rangle.
\end{equation}
Notice that
$     \langle \lambda_N^*, v \rangle
    = (1/N)\sum_{i=1}^N \langle \mu_{N,i}^*, v(\xi^i) \rangle
$.
Hence $\lambda_N^* \in \mathcal{C}(\Xi; R)^*$.

Now, we show that the conditions in \eqref{eq:KKTcontinuousmodel} imply
\begin{align}
\label{eq:KKTcontinuousmodel'}
    \begin{aligned}
& 0  \in    D\widehat{F}_N(u_N^*) + \partial \psi(u_N^*)+  D\mathcal{G}(u_N^*)^* \lambda_N^* ,\\
   & \lambda_N^*  \in \mathcal{K}^{-},
    \quad \langle \lambda_N^*, \mathcal{G}(u_N^*)\rangle = 0,  \quad
     \delta_{\xi^i}( \mathcal{G}(u_N^*)) \in K,
    \quad i = 1, \ldots, N.
    \end{aligned}
\end{align}

\begin{lemma}
\label{lemma:implied-KKT-conditions}
    If
\Cref{assumption:general-assumptions,assumption:continuous-setting} hold,
    the conditions in \eqref{eq:KKTcontinuousmodel} imply
    \eqref{eq:KKTcontinuousmodel'}.
\end{lemma}

\begin{proof}
    Using $\langle \mu_{N,i}^*, G(u_N^*,\xi^i) \rangle = 0$, 
    $i=1, \ldots, N$, and \eqref{eq:identification-lambda-mu}, we obtain the identity 
    $\langle \lambda_N^*, \mathcal{G}(u_N^*)\rangle
    = (1/N) \sum_{i=1}^N \langle \mu_{N,i}^*, G(u_N^*,\xi^i) \rangle = 0$.
    Since $\mu_{N,i}^* \in K^{-}$, we have
    $\langle \mu_{N,i}^*, v \rangle \leq 0$ for all $v \in K$.
    Let $k \in \mathcal{K}$; then 
     $k(\xi) \in K$ for all $\xi \in \Xi$.
    This yields $\langle \lambda_N^*, k \rangle  = (1/N) \sum_{i=1}^N \langle \mu_{N,i}^*, k(\xi^i) \rangle \leq 0$, and hence, $\lambda_N^* \in \mathcal{K}^{-}$.
    For all $h \in U$,
    \begin{align*}
        \langle D\mathcal{G}(u_N^*)^* \lambda_N^*, h\rangle
         =
        \frac{1}{N}\sum_{i=1}^N
        \langle \mu_{N,i}^*, D\mathcal{G}(u_N^*)(\xi^i)h
        \rangle
         = \frac{1}{N} \sum_{i=1}^N
        \langle \mu_{N,i}^*, D_u G(u_N^*,\xi^i)h
        \rangle.
    \end{align*}
    Hence the stationary conditions are equivalent.
\end{proof}

\subsection{Consistency of SAA KKT points}
\label{sec:consistency-KKT}
The section's main result is the almost sure consistency of the SAA KKT points.
In preparation of this consistency result, we first show that the Lagrange multipliers in the SAA conditions are bounded.
\begin{lemma}
\label{lem:multipliers-are-bounded}
Suppose that
\Cref{assumption:general-assumptions,assumption:continuous-setting} are fulfilled with the additional assumption that $\mathcal{K}$ has a nonempty interior. Let $(u_N^*,\lambda_N^*)$ be a sequence of KKT points satisfying \eqref{eq:KKTcontinuousmodel'}. Further suppose that \wpone,  $(u_N^*)$
weakly$^*$ converges to $\bar u$ such that \eqref{eq:CQ-robust-nonempty-interior} holds for $u^*=\bar{u}$.
Then
$(\lambda_N^*)$ is bounded \wpone.
\end{lemma}
\begin{proof}
The SLLN ensures $(1/N) \sum_{i=1}^N L(\xi^i)\to \mathbb{E}[L(\xi)]$
as $N \to \infty$ \wpone,
where $L(\xi)$ is the Lipschitz constant of $\mathscr{J}(\cdot,\xi)$
(see \Cref{assumption:continuous-setting}).  We show that
$(\lambda_N^*)$
is bounded using this and the constraint qualification.

Owing to the constraint qualification \eqref{eq:CQ-robust-nonempty-interior}, there exists $\rho >0$ such that
\[
\mathcal{G}(\bar{u}) +D\mathcal{G}(\bar{u})(\hat{u}-\bar{u}) +\mathbb{B}_{\mathcal{C}(\Xi; R)}(0;\rho) \subset  \mathcal{K}.
\]
Due to sequential weak$^*$-to-strong continuity of the mapping $u \mapsto f(u)\coloneqq \mathcal{G}(u) +D\mathcal{G}(u)(\hat{u}-u)$, there exists $N_0 \in \mathbb{N}$ such that for all $N\geq N_0$, we have $f(u_N^*) \in f(\bar{u})+\mathbb{B}_{\mathcal{C}(\Xi; R)}(0;\rho/2)$. In other words, we have
\[
\mathcal{G}(u_N^*) +D\mathcal{G}(u_N^*)(\hat{u}-u_N^*) +\mathbb{B}_{\mathcal{C}(\Xi; R)}(0;\rho/2) \subset \mathcal{K}
\]
for all $N \geq N_0$. 
This in turn implies for $N\geq N_0$ that for any $r \in \mathbb{B}_{\mathcal{C}(\Xi; R)}(0;\rho/2)$, there exists $k_N \in \mathcal{K}$ such that
\begin{equation}
\label{eq:CQ-as-equation}
\mathcal{G}(u_N^*) + D\mathcal{G}(u_N^*)(\hat{u} -u_N^*)+r =k_N.
\end{equation}

Since $\lambda_N^* \in \mathcal{K}^-$, we have $\langle \lambda_N^*, k_N\rangle \leq 0 = \langle \lambda_N^*,\mathcal{G}(u_N^*)\rangle$,
where the equality above follows by  \eqref{eq:KKTcontinuousmodel'}. Now,
\begin{equation}
\label{eq:inequalities-for-bounding-lambdaN}
\begin{aligned}
    \langle \lambda_N^*, r\rangle  &= \langle \lambda_N^*, k_N-\mathcal{G}(u_N^*)\rangle - \langle \lambda_N^*, D\mathcal{G}(u_N^*)(\hat{u}-u_N^*)\rangle \\
    &\leq- \langle \lambda_N^*, D\mathcal{G}(u_N^*)(\hat{u}-u_N^*)\rangle.
\end{aligned}
\end{equation}
From \eqref{eq:KKTcontinuousmodel'}, we also have $- D\widehat{F}_N(u_N^*) -  D\mathcal{G}(u_N^*)^* \lambda_N^* \in \partial \psi(u_N^*)$
 and therefore there exists an element $\eta_N^* \in \partial \psi(u_N^*)$ such that
\begin{equation*}
- D\widehat{F}_N(u_N^*)-  D\mathcal{G}(u_N^*)^* \lambda_N^* = \eta_N^*.
\end{equation*}
By definition of the subdifferential and the requirement that $\psi\geq 0$, we have \[\psi(\hat{u}) \geq \psi(u_N^*)+\langle \eta_N^*,\hat{u}-u_N^*\rangle \geq  \langle\eta_N^*,\hat{u}-u_N^*\rangle.\]
From \eqref{eq:inequalities-for-bounding-lambdaN}, we therefore have for each $r \in \mathbb{B}_{\mathcal{C}(\Xi; R)}(0;\rho/2)$ that
\begin{align*}
 \langle \lambda_N^*,r\rangle &\leq \langle \eta_N^*+D\widehat{F}_N(u_N^*), \hat{u}-u_N^*\rangle \leq \psi(\hat{u})+ \langle D\widehat{F}_N(u_N^*),\hat{u}-u_N^*\rangle\\
 & \leq \psi(\hat{u})+\lVert D\widehat{F}_N(u_N^*)\rVert \lVert \hat{u}-u_N^* \rVert \leq C_0 < \infty
\end{align*}
for some $C_0>0$. This constant exists since $(u_N^*)$ is bounded; moreover, as $N\to \infty$,
$$\|D\widehat{F}_N(u_N^*)\|
\leq \frac{1}{N} \sum_{i=1}^N \lVert D_uJ(u_N^*,\xi^i)\rVert
\leq \frac{1}{N} \sum_{i=1}^N \|B\|L(\xi^i)
\to \|B\|\mathbb{E}[L(\xi)].
$$
Finally, $(\lambda_N^*)$ is bounded, since
\[
\lVert \lambda^*_N\rVert= \frac{2}{\rho} \sup_{r \in \mathbb{B}_{\mathcal{C}(\Xi;R)}(0;\rho/2)}\langle \lambda_N^*,r\rangle \leq \frac{2C_0}{\rho} <\infty
\]
for all $N \geq N_0$.
\end{proof}

The next result establishes the consistency of SAA KKT points.

\begin{theorem}
\label{thm:KKT-system-is-consistent}
 Let the assumptions of \Cref{lem:multipliers-are-bounded} be satisfied with the sequence $(u_N^*,\lambda_N^*)$ such that $u_N^* \rightharpoonup^* \bar{u}$ \wpone. Then, \wpone, weak$^*$ limits of  $(u_N^*,\lambda_N^*)$ are KKT points of \eqref{eq:reduced-true}.
\end{theorem}

\begin{proof}
From \Cref{lem:objectivesepiconvergent}\ref{eq:epiconvergenceF+I}, 
$
    \mathbb{E}[I_{\mathcal{U}(\xi)}(\bar u)]
    \leq
    \liminf_{N \to \infty}\, (1/N)\sum_{i=1}^N I_{\mathcal{U}(\xi^i)}(u_N^*) = 0
$.
Hence $\bar u$ is feasible for \eqref{eq:reduced-true}.

From \Cref{lem:multipliers-are-bounded}, $(\lambda^*_N) \subset \mathcal{C}(\Xi; R)^*$
    is bounded, where $\lambda_N^*$ is the approximate multiplier in \eqref{eq:KKTcontinuousmodel'}.
    Since $\mathcal{C}(\Xi; R)$ is separable, the closed unit ball 
    is weakly\textsuperscript{$*$} sequentially compact.
    Hence there exists $\bar \lambda \in \mathcal{C}(\Xi; R)^*$
    such that $\lambda_N^* \rightharpoonup^* \bar \lambda$ as $N \to \infty$
    (at least along a subsequence). Owing to sequentially weak$^*$-to-strong continuity of $\mathcal{G}$, we have  $0 = \langle \lambda_N^*, \mathcal{G}(u_N^*)\rangle \to \langle \bar \lambda, \mathcal{G}(\bar u)\rangle$.
    Similarly, for each $h_N \rightharpoonup^* h$,
    \[\langle D\mathcal{G}(u_N^*)^* \lambda_N^*, h_N \rangle = \langle  \lambda_N^*, D\mathcal{G}(u_N^*) h_N \rangle\to \langle  \bar\lambda, D\mathcal{G}(\bar u) h \rangle=
    \langle D\mathcal{G}(\bar u)^* \bar \lambda, h \rangle \quad \text{as } N\rightarrow \infty \] since $(u,h) \mapsto D\mathcal{G}(u)h
    = D_w\mathscr{G}(Bu, \cdot)Bh$ is sequentially weakly$^*$-to-strongly continuous.
     Since $\lambda_N^* \in \mathcal{K}^{-}$
     and $\mathcal{K}^{-}$ is weakly$^*$ sequentially closed
     (cf.\ \cite[p.~31]{Bonnans2013}), we have
    $\bar \lambda \in \mathcal{K}^{-}$.

Using the uniform SLLN (see \cite[Corollary 4:1]{LeCam1953}),
\wpone,
\begin{align*}
    (w_1,w_2) \mapsto \frac{1}{N} \sum_{i=1}^N \langle D_w\mathscr{J}(w_1,\xi^i), w_2 \rangle
    \;\; \text{converges to} \;\;
    (w_1,w_2) \mapsto \mathbb{E}[\langle D_w\mathscr{J}(w_1, \xi) , w_2 \rangle]
\end{align*}
uniformly on the compact set $B(\mathrm{dom}(\psi)) \times B(\mathrm{dom}(\psi))$
as $N \to \infty$.

Using \eqref{eq:KKTcontinuousmodel'}
and the chain rule, we have
for all $u \in \mathrm{dom}(\psi)$,
\begin{align*}
    \psi(u) 
    &\geq 
    \psi(u_N^*) -
    \langle 
    D\widehat{F}_N(u_N^*)  +D\mathcal{G}(u_N^*)^* \lambda_N^*, 
    u-u_N^* \rangle 
    \\
    &= 
    \psi(u_N^*) 
    -
    \frac{1}{N} \sum_{i=1}^N\langle  
    D_w\mathscr{J}(Bu_N^*,\xi^i), B(u-u_N^*) \rangle
    -
    \langle D\mathcal{G}(u_N^*)^* \lambda_N^*, 
    u-u_N^* 
    \rangle.
\end{align*}
For the first expression on the right-hand side, 
we have
$\liminf_{N \to \infty}\, \psi(u_N^*)
\geq \psi(\bar u)$.
The second term converges to 
$\mathbb{E}[\langle D_w\mathscr{J}(B\bar u, \xi) , B(u-\bar u) \rangle]$
because of
the uniform SLLN, the continuity of $D_w\mathscr{J}(\cdot,\xi)$,
as well as the dominated convergence theorem.
The convergence of the third term 
to $\langle D\mathcal{G}(\bar u)^* \bar \lambda, u-\bar u \rangle$
was established above.

In summary, we have
for all $u \in \mathrm{dom}(\psi)$,
\begin{align*}
    \psi(u) 
    &\geq \psi(\bar u)
    - \mathbb{E}[\langle D_w\mathscr{J}(B\bar u, \xi) , B(u-\bar u) \rangle]
    - 
    \langle D\mathcal{G}(\bar u)^* \bar \lambda, u-\bar u \rangle
    \\
    &\geq \psi(\bar u)
    -\langle 
     DF(\bar u)  +D\mathcal{G}(\bar u)^* \bar \lambda, 
    u-\bar u \rangle.
\end{align*}
We obtain
$
        0 \in DF(\bar u)  + D \mathcal{G}(\bar u)^*\bar{\lambda} + \partial \psi(\bar u).
$
Hence, $(\bar u, \bar \lambda)$ is a KKT point of \eqref{eq:reduced-true}.
   \end{proof}

The above result showed consistency of the KKT points without requiring optimality of the sequence $(u_N)$. In the case where \eqref{eq:reduced-true} is  convex, \Cref{cor:necessary-and-sufficient} implies that the limit point $\bar{u}$ obtained in \Cref{thm:KKT-system-is-consistent} is in fact optimal for \eqref{eq:reduced-true}.

\subsection{Consistency of SAA KKT points with regularization}
\label{sec:SAA-KKT-MY}
Now, we turn to showing the consistency of SAA KKT conditions resulting from  regularization of the constraint (as the parameter $\gamma_N$ and the sample size $N$ are taken to infinity). This type of approach has been used, for instance, in~\cite{Gahururu2022}.

To show consistency, we strengthen \Cref{assumption:penalty} to the following.
\renewcommand{\theassumption}{2'}
\begin{assumption}
\label{assumption:penalty-MY}
    The mapping $\beta \colon R  \to [0,\infty)$
    is convex, continuously differentiable, and satisfies $\beta(r) = 0$ if and only if $r \in K$.
\end{assumption}
\renewcommand{\theassumption}{\arabic{assumption}}

Under \Cref{assumption:continuous-setting,assumption:penalty-MY}, $u \mapsto \varphi_\xi(u)\coloneqq \beta(G(u,\xi))$ is  continuously differentiable for any $\xi$, and on any bounded set in $U_0$, $D \varphi_{\xi}$ is uniformly bounded for all $\xi$ (in the compact set $\Xi$).  
In particular,  the derivative and the expectation can be exchanged (see \cite[Lemma C.3]{geiersbach_stochastic_2021}) so that
\begin{align*}
D\E[\varphi_\xi(u)] &= \E[D\varphi_\xi(u)] = \E[D_u G(u,\xi)^*D\beta(G(u,\xi))].
\end{align*}
Therefore, the optimality conditions for \eqref{eq:stateconstrainedproblem-robust-MY} with $\gamma=\gamma_N$ are given by
\begin{equation*}
    0 \in DF(u^{\gamma_N})+ \partial \psi(u^{\gamma_N})+ {\gamma_N} \E[D_u G(u^{\gamma_N},\xi)^*D\beta(G(u^{\gamma_N},\xi))].
\end{equation*}
The corresponding SAA optimality conditions are given by
\begin{equation*}
0 \in D\widehat{F}_N(u_N^{\gamma_N})+\partial \psi(u_N^{\gamma_N})+\frac{1}{  N} \sum_{i=1}^N D_u G(u_N^{\gamma_N},\xi^i)^*\mu^{{\gamma_N}}_{N,i},
\end{equation*}
where
\begin{equation}
\label{eq:KKT-MY-preliminary-version2}
\mu_{N,i}^{\gamma_N} \coloneqq {\gamma_N} D\beta(G(u_N^{\gamma_N},\xi^i)).
\end{equation}
Using the same arguments as in \Cref{lemma:implied-KKT-conditions}, one can show that these conditions imply
\begin{equation}
\label{eq:KKT-MY-continuous-version}
0 \in D\widehat{F}_N(u_N^{\gamma_N})+\partial \psi(u_N^{\gamma_N})+ D \mathcal{G}(u_N^{\gamma_N})^*\lambda^{{\gamma_N}}_{N}
\end{equation}
with $\lambda_N^{{\gamma_N}}$ defined by
\begin{equation}
\label{eq:identification-lambda-mu-MY}
    \langle \lambda_N^{\gamma_N}, v \rangle  \coloneqq \frac 1 N \sum_{i=1}^N \langle \mu_{N,i}^{{\gamma_N}}, \delta_{\xi^i}(v) \rangle
\end{equation}
for $v \in \mathcal{C}(\Xi;R)$, analogously to \eqref{eq:identification-lambda-mu}.
Similarly to \Cref{lem:multipliers-are-bounded}, we can show boundedness of the ``Lagrange multiplier'' $\lambda^{\gamma_N}_{N}$.
\begin{lemma}
\label{lem:multipliers-are-bounded-MY}
Suppose that
\Cref{assumption:general-assumptions,assumption:continuous-setting,assumption:penalty-MY} are satisfied and that $\mathcal{K}$ has a nonempty interior. Let $(u_N^{\gamma_N},\lambda_N^{\gamma_N})$ be a sequence of KKT points satisfying \eqref{eq:KKT-MY-continuous-version}, where $(\gamma_N) \subset (0,\infty)$
    is a sequence
    with $\gamma_N \to \infty$ as $N \to \infty$. Further suppose that, \wpone,  the sequence $(u_N^{\gamma_N})$ 
weakly$^*$ converges to $\bar u$ such that \eqref{eq:CQ-robust-nonempty-interior} is satisfied for $u^*=\bar{u}$.
Then
$(\lambda_N^{\gamma_N})$ is bounded in $\mathcal{C}(\Xi;R)^*$ \wpone.
\end{lemma}
\begin{proof}
The proof follows the arguments of \Cref{lem:multipliers-are-bounded}; we highlight the parts where a modification is needed. From the constraint qualification \eqref{eq:CQ-robust-nonempty-interior}, we have as with \eqref{eq:CQ-as-equation} the existence of $\rho>0$, $N_0 \in \mathbb{N}$ such that for a fixed $r\in \mathbb{B}_{\mathcal{C}(\Xi; R)}(0;\rho/2)$, there exists $k_N \in \mathcal{K}$ such that
\begin{equation}
\label{eq:eq:CQ-as-equation-MY}
\mathcal{G}(u_N^{\gamma_N})+D\mathcal{G}(u_N^{\gamma_N})(\hat{u}-u_N^{\gamma_N})+ r=k_N, \quad N\geq N_0.
\end{equation}
For each $\xi \in \Xi$, we note that $D \beta(k_N(\xi)) = 0$ since $\beta$ is smooth and attains its minimum value for $k_N(\xi) \in K$. Therefore, using the monotonicity of $D\beta$
and $k_N(\xi) \in K$ for all $\xi \in \Xi$, 
we have for all $N \in \mathbb{N}$,   $i\in \{1, \dots, N\}$,
\begin{align*}
&\gamma_N\langle D\beta(G(u_N^{\gamma_N},\xi^i))-D\beta(k_N(\xi^i)), G(u_N^{\gamma_N},\xi^i)- k_N(\xi^i) \rangle \\
&\quad= \langle \gamma_N D\beta(G(u_N^{\gamma_N},\xi^i)), G(u_N^{\gamma_N},\xi^i)-k_N(\xi^i)\rangle\geq 0.
\end{align*}
Applying \eqref{eq:identification-lambda-mu-MY},
and using $k_N(\xi^i) \in K$, we have
\[
\langle \lambda_N^{\gamma_N}, \mathcal{G}(u_N^{\gamma_N})-k_N\rangle \geq 0.
\]
From this point, arguments can be repeated using the definition of $r$ in \eqref{eq:eq:CQ-as-equation-MY}, since analogously to \eqref{eq:inequalities-for-bounding-lambdaN}, we have
\begin{align*}
\begin{aligned}
    \langle \lambda_N^{\gamma_N}, r\rangle  &= \langle \lambda_N^{\gamma_N}, k_N-\mathcal{G}(u_N^{\gamma_N})\rangle - \langle \lambda_N^{\gamma_N}, D\mathcal{G}(u_N^{\gamma_N})(\hat{u}-u_N^{\gamma_N})\rangle \\
    &\leq- \langle \lambda_N^{\gamma_N}, D\mathcal{G}(u_N^{\gamma_N})(\hat{u}-u_N^{\gamma_N})\rangle.
\end{aligned}
\end{align*}
\end{proof}

We can also obtain consistency as in \Cref{thm:KKT-system-is-consistent}.
\begin{theorem}
\label{thm:MY-KKT-system-is-consistent}
Let the assumptions of \Cref{lem:multipliers-are-bounded-MY} be satisfied with the sequence $(u_N^{\gamma_N},\lambda_N^{\gamma_N})$ such that $u_N^{\gamma_N} \rightharpoonup^* \bar{u}$. Then, \wpone, weak$^*$ limits of $(u_N^{\gamma_N},\lambda_N^{\gamma_N})$ are KKT points of \eqref{eq:reduced-true}.
\end{theorem}
\begin{proof}
 Let us  denote by $(\bar{u},\bar{\lambda})$ the weak$^*$ limit point of a given subsequence $(u_N^{\gamma_N},\lambda_N^{\gamma_N})_{\mathcal{N}}$. Owing to the weak$^*$ sequential closedness, $\bar{u} \in \mathrm{dom}(\psi)$.

 Using the same arguments as those  made in  \Cref{thm:KKT-system-is-consistent}, we can show that
\begin{equation*}
0 \in DF(\bar{u}) +  D\mathcal{G}(\bar{u})^* \bar{\lambda}+\partial \psi(\bar{u}).
\end{equation*}
It remains to prove that the conditions $\mathcal{G}(\bar{u}) \in \mathcal{K}$, $\bar{\lambda} \in \mathcal{K}^-$, and $\langle \bar{\lambda},\mathcal{G}(\bar{u})\rangle = 0$ are satisfied. To this end, it is enough to argue that $\bar{\lambda} \in N_\mathcal{K}(\mathcal{G}(\bar{u}))$; see \cite[eq.\ (3.9)]{Bonnans2013}.
On account of the convexity and differentiability of $\beta$, we have for any $r,r' \in R$
\[
\beta(r) \geq \beta(r') +\langle D\beta(r'),r-r'\rangle.
\]
In particular, with $r' = G(u_N^{\gamma_N},\xi)$, we have for all $N \in \mathbb{N}$ and all $i \in \{1, \dots, N\}$ that
\[
0 \leq \gamma_N \beta(G(u_N^{\gamma_N},\xi^i)) \leq \gamma_N \beta(r) - \gamma_N \langle D \beta(G(u_N^{\gamma_N},\xi^i)), r-G(u_N^{\gamma_N},\xi^i)\rangle.
\]
Recalling the definition of $\mu_{N,i}^{\gamma_N}$ given in \eqref{eq:KKT-MY-preliminary-version2}, it follows for any $r \in K$ that
\[
\langle \mu_{N,i}^{\gamma_N}, r-G(u_N^{\gamma_N},\xi^i)\rangle \leq 0.
\]
Let $k \in \mathcal{K}$; then, we have $k(\xi) \in K$ for all $\xi \in \Xi$.
Therefore, for all $N \in \mathbb{N}$, we have
\[\langle \lambda_N^{\gamma_N}, k-\mathcal{G}(u_N^{\gamma_N})\rangle = \frac{1}{N}\sum_{i=1}^N \langle \mu_{N,i}^{\gamma_N}, k(\xi^i)-G(u_N^{\gamma_N},\xi^i)\rangle \leq 0.\]

Taking limits as $\mathcal{N}\ni N\rightarrow\infty$ and using the sequential weak$^*$-to-strong continuity of $u\mapsto \mathcal{G}(u)$, we have $\bar{\lambda} \in N_\mathcal{K}(\mathcal{G}(\bar{u}))$ as claimed.
\end{proof}

\section{Sample complexity of SAA Moreau--Yosida solutions}
\label{sec:sample-complexity-MY}
The main goal of the section is to provide informed choices of the
penalty parameter $\gamma_N$ as a function of the sample size $N$.
The parameter must be large enough to ensure approximate feasibility of  penalized problem's solutions 
but larger values increase the variance of the SAA Moreau--Yosida solutions---requiring a compromise.
This is similar to a bias-variance trade-off. Our first step is to define
an error measure that captures this trade-off and then derive an upper bound
on the error measure. Minimizing this upper bound over $\gamma_N > 0$ results
in an informed  choice of $\gamma_N$.

We consider the 
error measure
\begin{align*}
    \Phi(s) \coloneqq 
    \min_{u \in U}\, \max\{F(u)+\psi(u)-s, \mathbb{E}[\beta(G(u,\xi))]\},
\end{align*}
which has been used in \cite[Section 2.2]{Guigues2017} 
as an error indicator for stochastic optimization with expectation constraints, for example.
We have $\Phi \geq 0$, and $\Phi(s^*) = 0$ for $s^* = F(u^*) + \psi(u^*)$,
 where $u^*$ solves \eqref{eq:stateconstrainedproblem}.

The following sample complexity result applies to optimization problems that satisfy \Cref{assumption:general-assumptions,assumption:continuous-setting,assumption:penalty-MY}, along with additional conditions stated explicitly in the proposition.
Let \Cref{assumption:general-assumptions,assumption:continuous-setting,assumption:penalty-MY} hold. We recall from \cref{sec:MY-solutions-consistency} that $ u_N^{\gamma_N} $ is a 
solution to the SAA penalty problem \eqref{eq:penaltyproblem}. Similarly, its true counterpart 
\eqref{eq:stateconstrainedproblem-robust-MY}
has a solution $u^{\gamma}$.
For (Bochner) integrable random elements $Z$ with values in $U$,
we define $\mathrm{Var}(Z) \coloneqq \mathbb{E}[\|Z-\mathbb{E}[Z]\|_U^2]$.
From \cref{sec:SAA-KKT-MY}, we recall $\varphi_\xi(u) = \beta(G(u,\xi))$.

\begin{proposition}
\label{prop:samplecomplexity-penalty-approach}
Suppose that
\Cref{assumption:general-assumptions,assumption:continuous-setting,assumption:penalty-MY} are satisfied. Additionally, let $\gamma_N > 0$,
let $U$ be a separable Hilbert space, let $J(\cdot, \xi)$
be convex and let $G(\cdot,\xi)$ be affine-linear for each $\xi \in \Xi$, 
and
let 
$\mathrm{Var}(\nabla_u J(u^{\gamma_N},\xi))$
and
$\mathrm{Var}(\nabla_u \varphi_\xi(u^{\gamma_N}))$
be finite.
Moreover, suppose that $\psi$ is
strongly convex with parameter $\alpha >0$, and
that $F+\psi$ is Lipschitz continuous on $\mathrm{dom}(\psi)$
with Lipschitz constant $L_{F+\psi} >0$.
Then 
\begin{align*}
    &\mathbb{E}\big[\Phi\big(F(u_N^{\gamma_N}) + \psi(u_N^{\gamma_N})\big)\big] \leq  
    (1/\gamma_N) (F(u^*) + \psi(u^*))
    \\
      &  \quad\quad \quad  + \tfrac{2L_{F+\psi}}{\alpha N^{1/2}}
     \Big[
    \big(\mathrm{Var}(\nabla_u J(u^{\gamma_N},\xi))\big)^{1/2}
    + 
    \gamma_N
    \big(\mathrm{Var}(\nabla_u \varphi_\xi(u^{\gamma_N}))\big)^{1/2}
    \Big].
\end{align*}
\end{proposition}
\begin{proof}
\Cref{assumption:general-assumptions,assumption:continuous-setting,assumption:penalty-MY} 
ensure that  $J(\cdot, \xi)$ and
$\varphi_\xi(\cdot) = \beta(G(\cdot,\xi))$ are continuously differentiable
and their gradients
evaluated at $u^{\gamma_N}$
are integrable, as shown in \cref{sec:KKT-conditions-original,sec:SAA-KKT-MY}. 
Combining the fact that $U$ is a separable
Hilbert space with the imposed convexity assumptions,
we can apply \cite[Theorem~3]{Milz2023}
to the $\alpha$-strongly convex integrand
$(u,\xi) \mapsto 
J(u,\xi)
+ \gamma_N \beta (G(u,\xi))
+\psi(u)
$ and obtain 
\begin{align}
\label{eq:penalty-expectation-bound}
    \mathbb{E}[\|u_N^{\gamma_N} - u^{\gamma_N}\|_{U}^2]
    &\leq \frac{2\mathrm{Var}(\nabla_u J(u^{\gamma_N},\xi))}{\alpha^2N}
    + 
    \frac{2\gamma_N^2 \mathrm{Var}(\nabla_u [\beta (G(u^{\gamma_N},\xi))])}{\alpha^2N}.
\end{align}

For all $s \in \mathbb{R}$, 
$u^{\gamma_N} \in U$
and $\beta \geq 0$ ensure
\begin{align*}
    \Phi(s)
    &\leq  
    |F(u^{\gamma_N}) + \psi(u^{\gamma_N}) 
    - s|
    + 
    \mathbb{E}[\beta(G(u^{\gamma_N},\xi))].
\end{align*}
Since $F$ is nonnegative, $\gamma_N >0$, $u^{\gamma_N}$
solves the penalty problem, and $u^*$ is feasible for the true problem, we obtain
$\mathbb{E}[\beta(G(u^{\gamma_N},\xi))] 
\leq (1/\gamma_N) (F(u^*) + \psi(u^*))$. Hence
\begin{align*}
     \Phi(F(u_N^{\gamma_N}) + \psi(u_N^{\gamma_N}))
     & \leq |F(u^{\gamma_N}) + \psi(u^{\gamma_N}) -[F(u_N^{\gamma_N}) + \psi(u_N^{\gamma_N})]|
     \\ & \quad + (1/\gamma_N) (F(u^*) + \psi(u^*)).
\end{align*}

Together with the expectation bound \eqref{eq:penalty-expectation-bound}, 
and the Lipschitz continuity of $F+\psi$, we obtain the assertion.
\end{proof}

Minimizing $\gamma \mapsto (1/\gamma) + (\gamma/N^{1/2})$ over $\gamma > 0$ yields the solution
$\gamma = N^{1/4}$. Ignoring problem-dependent constants
and the dependence of the variance terms on $u^{\gamma_N}$, 
\Cref{prop:samplecomplexity-penalty-approach}
suggests choosing $\gamma_N$ proportional to $N^{1/4}$.

\section{Approximate feasibility of SAA solutions}
\label{sec:approximate-feasibility-SAA}
Solutions obtained through the SAA problem \eqref{eq:stateconstrainedproblem-SAA}
are generally not feasible for the original problem \eqref{eq:stateconstrainedproblem}. 
Beyond the standard challenges in analyzing feasibility when dealing with random feasible sets, as discussed in \cite{Calafiore2005,Campi2008}, a significant difficulty arises from the fact that the domain of the function $\psi$ is typically noncompact. To overcome this, we instead work with the compact set $B(\mathrm{dom}(\psi))$.

In what follows, we present a basic result that is motivated by \cite[Theorem 10]{Luedtke2008}. For a comprehensive analysis of feasibility in the case of convex uncertain constraints, we refer the reader to \cite{Calafiore2005,Campi2008}. For recent developments that extend such results to settings involving nonconvex constraints, we refer the reader to \cite{Garatti2025}.

We define 
$\mathscr{U}_{\varepsilon}(\rho) \coloneqq 
\{u \in \mathrm{dom}(\psi) \mid \mathbb{P}(\mathscr{G}(Bu,\xi) \in K
+ \bar{\mathbb{B}}_R(0;\rho))
\geq 1-\varepsilon\}
$,
where
$\varepsilon \in (0,1)$.
For a  compact metric space $\mathcal{Y}$,
let $\mathcal{N}(\nu, \mathcal{Y})$
be the $\nu$-covering number
with respect to the $\mathcal{Y}$-metric.

\begin{proposition}
Suppose that
\Cref{assumption:general-assumptions,assumption:continuous-setting,assumption:penalty-MY} are satisfied. 
If there exists
a constant $L_{\mathscr{G}} > 0$
such that
for all $\xi \in \Xi$,
$\mathscr{G}(\cdot,\xi)$,
considered as a mapping from $
B(\mathrm{dom}(\psi))$ to $R$,
is Lipschitz continuous with Lipschitz
constant $L_{\mathscr{G}} > 0$,
and
$\varepsilon > 0$,
$\rho > 0$,
$\delta \in (0,1)$, as well as
\begin{align*}
	N \geq \frac{1}{\varepsilon}
	\bigg[
	 \ln\Big(\frac{1}{\delta}\Big) +
	 \ln \mathcal{N}\Big(
	 \frac{\rho}{2L_{\mathscr{G}}}, B(\mathrm{dom}(\psi))
	 \Big)
	\bigg],
\end{align*}
then with a probability of at least $1-\delta$,
the feasible set of \eqref{eq:stateconstrainedproblem-SAA} 
is contained in $\mathscr{U}_{\varepsilon}(\rho)$.
\end{proposition}
\begin{proof}
Let $\rho^\prime \coloneqq \rho/(2L_{\mathscr{G}})$.
Since $B(\mathrm{dom}(\psi))$ is compact
according to \Cref{lem:properties-B},
$Q \coloneqq \mathcal{N}(\rho^\prime, 
B(\mathrm{dom}(\psi)))$ is finite.
In particular, there exist
$\bar w_1, \ldots, \bar w_Q \in B(\mathrm{dom}(\psi))$ 
such that for all
$w \in B(\mathrm{dom}(\psi))$, 
there exists $q(w) \in \{1, \ldots, Q\}$
such that
$\|w-\bar w_{q(w)}\| \leq \rho^\prime$.
Since $\bar w_q \in B(\mathrm{dom}(\psi))$,
there exists, by definition
of image sets, 
an element $\bar u_q \in \mathrm{dom}(\psi)$
such that $\bar w_q = B \bar u_q$, $q = 1, \ldots, Q$.
We define
\[
\mathscr{U}_N \coloneqq 
\{
u \in \{\bar u_1, \ldots, \bar u_Q\} \mid 
\mathscr{G}(Bu,\xi^i)
\in K + \bar{\mathbb{B}}_R(0;L_{\mathscr{G}} \rho^\prime), \quad 
i = 1, \ldots, N
\}.
\]
Since the set $\{\bar u_1, \ldots, \bar u_Q\}$ contains at most $Q$ elements, we can apply the Bernoulli trials result from \cite[Theorem 7]{Luedtke2008}, which remains valid even in infinite-dimensional decision spaces. This yields, with probability at least
$1 - Q (1-\varepsilon)^N$, 
\[
\mathscr{U}_N 
\subset 
\{\bar u_1, \ldots, \bar u_Q\} \cap \mathscr{U}_{\varepsilon}(\rho/2)
\subset \mathscr{U}_{\varepsilon}(\rho/2).
\]

Now, let $u_N$ be feasible for the SAA problem, 
that is, $u_N$ satisfies
$u_N \in \mathrm{dom}(\psi)$
and
$\mathscr{G}(Bu_N,\xi^i) \in K$, 
$i = 1, \ldots, N$.
Since $Bu_N \in B(\mathrm{dom}(\psi))$, there
exists $q \in \{1, \ldots, Q\} $
such that
$\|Bu_N-B\bar u_q\|\leq \rho^\prime$. 
Combined with the $L_{\mathscr{G}}$-Lipschitz continuity of 
$\mathscr{G}(\cdot,\xi)$,
which ensures
$\mathscr{G}(B\bar u_q,\xi)- \mathscr{G}(Bu_N,\xi)
\in \bar{\mathbb{B}}_R(0;L_{\mathscr{G}} \rho^\prime)$
for all $\xi \in \Xi$, 
we have
\begin{align*}
    \mathscr{G}(B\bar u_q,\xi^i)
    \in 
    \mathscr{G}(Bu_N,\xi^i)
    + \bar{\mathbb{B}}_R(0;L_{\mathscr{G}} \rho^\prime)
    \in K + \bar{\mathbb{B}}_R(0;L_{\mathscr{G}} \rho^\prime),
    \quad i=1, \ldots, N.
\end{align*}
Hence $\bar u_q \in \mathscr{U}_N$. This ensures
with a probability of at least $1 - Q (1-\varepsilon)^N$, 
$\bar u_q \in \mathscr{U}_{\varepsilon}(\rho/2)$ and hence
$u_N \in \mathscr{U}_{\varepsilon}(\rho)$
because of
$\|Bu_N-B\bar u_q\|\leq \rho^\prime$ and
the $L_{\mathscr{G}}$-Lipschitz continuity of
$\mathscr{G}(\cdot,\xi)$ for all $\xi \in \Xi$.

As a consequence of the above
two paragraphs,
with a probability of at least
$1 - Q (1-\varepsilon)^N$, 
the SAA feasible set is contained
in $\mathscr{U}_{\varepsilon}(\rho)$.
Since $\ln(Q/\delta) =   \ln(1/\delta) + \ln(Q)$, the hypothesis
on $N$ ensures
$N \geq (1/\varepsilon) \ln(Q/\delta)$, 
yielding
$1 - Q \mathrm{e}^{-\varepsilon N}
\geq 1- \delta$.
Hence
$1 - Q (1-\varepsilon)^N
\geq 1 - \delta$.
This concludes our verification.
\end{proof}

\section{Applications}
\label{sec:applications}
In this section, we will show that our framework applies to a large class of infinite-dimensional problems ranging from learning to optimal control. In each example, we will demonstrate that \Cref{assumption:general-assumptions} is satisfied.
If not stated otherwise, we let $\Xi$ be the support of $\mathbb{P}$ throughout the section.

\subsection{Nonparametric regression in Sobolev spaces}

We start with an example from nonparametric regression, 
formally showing that \Cref{example:regression} is an instance of our problem formulation.
Motivated by
\cite[Example~3]{Cucker2002}, it involves least squares
regression in a Sobolev space with an additional
inequality constraint.

The unknown regression function is defined on  a bounded Lipschitz domain $\domain\subset\mathbb{R}^d$. We assume that $\rho> 0$, 
$s > d/2$ is an integer,
$\Xi \coloneqq \bar \domain \times [-1,1]$,
and
$\xi \coloneqq (x,y)$.
Let 
$\psi(u) \coloneqq I_{\Uad}(u)$,
where $\Uad \coloneqq 
\{u \in H^s(\domain) \mid \|u\|_{H^s(\domain)} \leq \rho\}$. We consider
\begin{align*}
    \min_{u \in H^s(\domain)}\, 
    (1/2)\mathbb{E}[([Bu](x)-y)^2] + \psi(u)
    \quad \text{s.t.} \;\; [Bu](x) \geq 0 \;\; \text{for all} \;\; x \in \bar \domain,
\end{align*}
where 
$B \colon H^s(\domain) \to \mathcal{C}(\bar \domain)$
is the embedding operator.

We now verify \Cref{assumption:general-assumptions}.

\paragraph{Function spaces} 
We define the control space
$U \coloneqq H^s(\domain)$, and
the image spaces $W \coloneqq \mathcal{C}(\bar \domain)$,
$R \coloneqq R_0 \coloneqq \mathbb{R}$.
Since $\bar \domain$ is compact, $W$ is separable.
The space $U$ is a separable Hilbert space.

\paragraph{Linear operator}
Since $B$ is linear and compact
(see \cite[Theorem~1.14]{Hinze2009}),
and $U$ is a Hilbert space, $B$ is weakly-to-strongly continuous.

\paragraph{Objective}
We define $\mathscr{J} \colon \mathcal{C}(\bar \domain) \times \Xi \to \mathbb{R}$
by
$\mathscr{J}(w,(x,y)) \coloneqq  (1/2)([w](x)-y)^2$.
Since $\mathcal{C}(\bar \domain) \times \bar \domain \ni (w, x) \mapsto [w](x)$
is continuous,  $\mathscr{J}$
is continuous. In particular, it is random lower semicontinuous.

\paragraph{Constraints}
The constraint mapping 
$\mathscr{G} \colon \mathcal{C}(\bar \domain) \times \Xi \to \mathbb{R}$
defined by
$\mathscr{G}(w, (x,y)) \coloneqq [w](x)$
is continuous and hence Carath\'eodory.
We also let
$K \coloneqq [0,\infty)$.

\paragraph{Control regularizer}
The set $\Uad$ 
is nonempty, closed, convex, and bounded.

\paragraph{Feasibility and finite objective value}
The function $u = 0$ is feasible
and the support of $y$ is contained in $[-1,1]$.
Hence, the mean of $\mathscr{J}(0,(x,y))
= (1/2)y^2$ is finite.

\subsection{Learning Hilbert--Schmidt operators}
We continue with an example from operator learning, partially 
motivated by 
regression
with operators  \cite{Nelsen2024}. 
The goal is to learn a Hilbert--Schmidt operator \(T\) from data
 by minimizing a least squares misfit between
predicted outputs and observations.
In our model, \(T\) is not accessed directly; instead, only a  transformation is observed. We model this by
introducing a compact map \(B\) on the space of Hilbert--Schmidt operators and expressing both
the loss and the pointwise cone constraint in terms of \(BT\) rather than \(T\).
The resulting formulation is a least squares problem over the space of Hilbert--Schmidt operators
subject to pointwise inequality constraints.

The set, $\mathrm{HS}(H)$, of Hilbert--Schmidt operators is defined on a real separable Hilbert space $H$.
We assume that $\rho > 0$, 
$\Xi \subset H \times H$, 
$\xi \coloneqq (x,y)$
with $\mathbb{E}[\|y\|_H^2] < \infty$.
Let  $\psi(T) \coloneqq I_{\mathscr{T}_{\text{ad}}}(T)$,
where $\mathscr{T}_{\text{ad}} \coloneqq 
\{T \in \mathrm{HS}(H) \mid \|T\|_{\mathrm{HS}(H)} \leq \rho\}$.
We consider 

\begin{align*}
    \min_{T \in \mathrm{HS}(H)}\, 
    (1/2)\mathbb{E}[
    \|BT x - y\|_H^2
    ]
    + \psi(T)
    \quad \text{s.t.}
    \quad BTx \in K
    \quad \text{for all} \quad (x,y) \in \Xi,
\end{align*}
where 
$K \subset H$ is a closed,
convex, nonempty cone.
Let $(a_j) \subset \mathbb{R}$
with 
$a_j \to 0$,
and let 
$(f_j)$ 
be an
orthonormal basis
in $H$.
We define
$B \colon \mathrm{HS}(H) \to \mathrm{HS}(H)$
such that for each
$T \in \mathrm{HS}(H)$, 
$B(T) \coloneqq A T A$,
where for each $x \in H$,
$Ax \coloneqq \sum_{j=1}^\infty
a_j(x, f_j)_H f_j$.
The operator $B$
is linear and  well-defined.
This definition is motivated
by the identity
$T x = \sum_{j=1}^\infty
(Tx, f_j)_H f_j
$.

We now verify \Cref{assumption:general-assumptions}.

\paragraph{Function spaces} 
We let $U \coloneqq W \coloneqq \mathrm{HS}(H)$.
This is a separable Hilbert space,
as $H$ is separable.
Moreover, $R \coloneqq R_0
\coloneqq H$.

\paragraph{Linear operator}
We show that
$B$ is compact by
demonstrating that
$B$ can be uniformly approximated by 
the
finite-rank operators
$B_M$
defined by
$
B_M (T) \coloneqq A_M T A_M
$,
where
$A_Mx \coloneqq \sum_{j=1}^M
a_j(x, f_j)_H f_j$.
Indeed, 
$
    \|B - B_M\|
    \leq 
    2\sup_{j \geq 1}\, 
    |a_j|
    \sup_{j > M}\, 
    |a_j|
    \to 0
$
as $M \to \infty$.

\paragraph{Objective}
We define 
$\mathscr{J} \colon \mathrm{HS}(H) \times \Xi \to \mathbb{R}$ by
$\mathscr{J}(w,(x,y)) \coloneqq  (1/2)\|[w](x) - y\|_H^2$,
which is  continuous.

\paragraph{Constraints}
The mapping $\mathscr{G} \colon \mathrm{HS}(H) \times \Xi \to H$
defined by
$\mathscr{G}(w, (x,y)) \coloneqq [w](x)$
is continuous and hence it is a Carath\'eodory map.

\paragraph{Control regularizer}
Since $\mathrm{HS}(H)$ is a Hilbert space and hence reflexive,
the  set $\mathscr{T}_{\text{ad}}$ is  weakly compact. It is also nonempty.

\paragraph{Feasibility and finite objective value}
The operator $T = 0$ is 
in $\mathscr{T}_{\text{ad}}$,
and $y$ has finite second moment.
Hence, the mean of $\mathscr{J}(0,(x,y))
= (1/2)\|y\|_H^2$ is finite.

\subsection{Learning Kantorovich potentials}
We now present an example from optimal transport, motivated by \cite{Chewi2025,Rachev1998}. 
In data-driven optimal transport, it is often convenient to work with the dual
optimal transport problem, since it can be approximated directly from samples of the
marginals.
We therefore present the dual formulation and subsequently show how it fits into our
framework.
We consider the dual optimal transport problem
(cf.\ \cite[Theorem~2.1.1]{Rachev1998} and \cite[eq.~(3.42)]{Friesecke2025}),
\begin{align}
\label{eq:dual}
\begin{aligned}
    & \max_{
    u_1, u_2 \colon M 
    \to \mathbb{R} \, 
    \text{bounded, measurable}
    }\, 
    \int_{M} u_1(x_1) \, \mathrm{d} P_1(x_1)
    +
    \int_{M} u_2(x_2) \, \mathrm{d} P_2(x_2)
    \\
    & \text{s.t.}
    \quad 
    u_1(x_1) + u_2(x_2) \leq  c(x_1, x_2)
    \quad 
    \text{for all}
    \quad 
    x_1, x_2 \in M,
\end{aligned}
\end{align}
where $M$ is a compact metric space
with a diameter of at most two, 
$c  \in \mathrm{Lip}(M \times M)$
is nonnegative,
and $P_1$ and $P_2$
are probability measures defined on $M$.
The assumption on the diameter of $M$ allows us to show
that the decision space $U$, which we define below, is the 
dual to a separable Banach space. 
Solutions to \eqref{eq:dual} are called
Kantorovich potentials
(see, e.g., \cite[p.\ 122]{Friesecke2025}).

Now, we reformulate the dual optimal transport problem as an instance of our problem formulation.
According to \cite[Corollary~3.18]{Ambrosio2024}, there exists a solution
$u^* \in \mathrm{Lip}(M)^2$
to \eqref{eq:dual}. 
Next, we show that this solution regularity allows us to reformulate
\eqref{eq:dual} as an instance of
\eqref{eq:stateconstrainedproblem}.
Let  $\psi(u) \coloneqq I_{\Uad}(u)$,
where $\Uad \coloneqq 
\{u \in \mathrm{Lip}(M)^2 \mid \|u\|_{\mathrm{Lip}(M)^2} \leq \rho\}$
with $\rho \geq \|u^*\|_{\mathrm{Lip}(M)^2}$.
We define
$\Xi \coloneqq M^2$, and
$\xi \coloneqq (x_1, x_2)$, 
and let 
$\mathbb{P}$ 
be the product probability measure of $P_1$ and $P_2$.
Consider the problem
\begin{align*}
    & \max_{
    (u_1,u_2) \in \mathrm{Lip}(M)^2}\, 
    \int_{M} [Bu_1](x_1) \, \mathrm{d} P_1(x_1)
    +
    \int_{M} [Bu_2](x_2) \, \mathrm{d} P_2(x_2)
    - \psi(u)
    \\
    & \text{s.t.}
    \quad 
    [Bu_1](x_1) + [Bu_2](x_2) \leq c(x_1, x_2)
    \quad 
    \text{for all}
    \quad 
    x_1, x_2 \in M,
\end{align*}
where
$B \colon \mathrm{Lip}(M)^2 \to \mathcal{C}(M)^2$ is the
embedding operator.
The dual optimal transport problem \eqref{eq:dual} is equivalent to 
this formulation owing to $u^* \in \mathrm{Lip}(M)^2$.

We now verify \Cref{assumption:general-assumptions}.

\paragraph{Function spaces} 
We define
$U \coloneqq \mathrm{Lip}(M)^2$, 
$W \coloneqq \mathcal{C}(M)^2$, and
$R \coloneqq R_0 \coloneqq \mathbb{R}$.
Since $M$ is compact, $W$ is separable.

Now, we show that $U$ is the dual to
a separable Banach space.
Let $M^+$ be the pointed extension of $M$ as defined in
\cite[p.~2]{Weaver1999}. According to 
\cite[Theorem~1.7.2]{Weaver1999}
and our assumptions on $M$, 
$\mathrm{Lip}(M)$ is isometrically isomorphic
with $\mathrm{Lip}_0(M^+)$.
For a pointed metric space $M_0$,
$\mathrm{Lip}_0(M_0)$ is the dual space to 
a Banach space
(see, e.g., 
\cite[Theorem~2.2.2]{Weaver1999}).
This predual is separable,
provided that $M_0$ is separable.
This separability can  be deduced from the definition of the predual
in \cite[Definition~2.2.1]{Weaver1999}.
Hence, $\mathrm{Lip}(M)$ as well as $U$ are duals to separable Banach spaces.

\paragraph{Linear operator}
Let $M_0$ be a pointed metric space.
Let $\iota \colon \mathrm{Lip}_0(M_0) \to \mathcal{C}(M_0)$ be the embedding operator.
We show that it is sequentially weakly$^*$-to-strongly
continuous.
If $u_N \rightharpoonup^* u$
in $\mathrm{Lip}_0(M_0)$,
then $u_N \to u$
in $\mathcal{C}(M_0)$
according to \cite[Proposition~2.1.7]{Weaver1999}.
Since $\iota$ is 
sequentially weakly$^*$-to-strongly continuous, 
and $\mathrm{Lip}(M) = \mathrm{Lip}_0(M^+)$,
it follows that $B$
is sequentially weakly$^*$-to-strongly continuous.

\paragraph{Objective}
We define 
$\mathscr{J} \colon \mathcal{C}(M)^2 \times M^2 \to [0,\infty)$
by
$\mathscr{J}(w, (x_1,x_2)) \coloneqq \max\{2\rho-[w_1](x_1) - [w_2](x_2) ,0\}$,
which is continuous. 

\paragraph{Constraints}
We define 
$\mathscr{G} \colon \mathcal{C}(M)^2 \times M^2 \to \mathbb{R}$
by
$\mathscr{G}(w, (x_1,x_2)) \coloneqq [w_1](x_1) + [w_2](x_2) - c(x_1,x_2)$,
which is continuous, 
and
$K \coloneqq (-\infty, 0]$.

\paragraph{Control regularizer}
Since $M^+$ is a pointed metric space,
and $\mathrm{Lip}(M)$ is isometrically isomorphic
to $\mathrm{Lip}_0(M^+)$,
which is the dual to a separable Banach space,
the norm $\|\cdot\|_{\mathrm{Lip}(M)}$
is sequentially weakly$^*$  lower semicontinuous
(cf.\ \cite[Remark~8.3 part (3) on p.\ 228]{Alt2016}).
Hence, $\Uad$ is sequentially weakly$^*$  compact.

\paragraph{Feasibility and finite objective value}
Since $c$ is nonnegative,
$u=0$ is feasible.
We also have $F(0) = 2\rho$.

\subsection{Optimization with dynamical systems}
We introduce a class of ODE-constrained optimization problems under uncertainty,
partly motivated by \cite{Melnikov2024,Ruppen1995},
with state constraints and a constraint on the control's 
bounded variation seminorm.
For example, in batch reactor optimization, state constraints capture operating
limits (e.g., temperature and pressure caps) and product specifications (e.g.,
bounds on concentrations).
The bounded-variation constraint is natural for switching-type actuation, as it aims to exclude
high-frequency chattering. It is also
compatible with known regularity results:  for certain problems, optimal controls
are of bounded variation even without such a constraint
\cite[Lemma~4.5]{Hager1979}. 
Analytically,  this bounded-variation
constraint allows us to use  compact embeddings from the BV space into Lebesgue spaces. Total variation regularization is also widely used in image denoising \cite{Rudin1992}.

The state is controlled on the time interval $[0,T]$, where $T > 0$ is finite.
We assume that $y_0 \in \mathbb{R}^n$,
$\ell \colon \mathbb{R}^n \to [0,\infty)$,
$f \colon \mathbb{R}^n \times \mathbb{R}^m \times  \Xi \to \mathbb{R}^n$, 
$h \colon \mathbb{R}^n  \to \mathbb{R}^{n}$.
We also assume that $\ell$ and $h$ are continuous, $f(y_0, 0, \xi) = 0$ for all $\xi \in \Xi$, and   
$h(y_0) \leq 0$.
Let $\psi(u) \coloneqq I_{\Uad}(u)$,
where $\Uad  \coloneqq \{ u \in \mathrm{BV}(0,T)^m \mid u(t) \in [\mathfrak{l}, \mathfrak{u}] \; \text{a.e.} \;t \in [0,T], \;
\|u_j'\|_{\mathcal{M}(0,T)} 
\leq a_j,
\; j = 1, \ldots, m
\}
$, where
$\mathfrak{l}$, $\mathfrak{u} \in \mathbb{R}^m$
with
$\mathfrak{l} < 0 < \mathfrak{u}$,
and $a \in \mathbb{R}^m$ with
$a > 0$. Consider 
\begin{align*}
\min_{u \in \mathrm{BV}(0,T)^m}\quad
&\mathbb{E}[\ell(y(T,\xi))]
+ \psi(u)
\end{align*}
with $y \colon [0,T] \times \Xi \to \mathbb{R}^n$  satisfying the  following constraints for all $\xi \in \Xi$:
\begin{align*}
    \frac{\du}{\du t}y(t,\xi) 
       & = f(y(t,\xi),u(t), \xi)
      \quad \text{a.e.} \quad 
      t \in [0,T],
      \quad y(0,\xi) = y_0, 
      \\
      h(y(t,\xi))  &\leq 0
      \quad \text{for all}
      \quad t \in [0,T].
\end{align*}

We assume that, for each $(w, \xi) \in L^2(0,T)^m \times \Xi$, the initial value  
problem 
\begin{align*}
    \frac{\du}{\du t}y(t,\xi) 
       & = f(y(t,\xi),w(t), \xi)
      \quad \text{a.e.} \quad 
      t \in [0,T],
      \quad y(0,\xi) = y_0,
\end{align*}
admits a unique solution 
$\mathscr{S}(w, \xi) \in W(0,T) \coloneqq 
\{ y \in L^2(0,T)^n  \mid  y^\prime \in L^2(0,T)^n \}$.  
Also, the operator $\mathscr{S} \colon L^2(0,T)^m \times \Xi \to W(0,T)$ is  
assumed to be Carath\'eodory.
These assumptions may be verified for particular classes of initial value problems at hand
(cf.\ the analysis developed in \cite{Scagliotti2023}).

We now verify \Cref{assumption:general-assumptions}.

\paragraph{Function spaces} 
We define
$U \coloneqq \mathrm{BV}(0,T)^m$, 
$W \coloneqq L^2(0,T)^m$,
$R \coloneqq R_0 \coloneqq \mathcal{C}([0,T]; \mathbb{R}^n)$.
Since $\mathrm{BV}(0,T)$ is the dual to 
a separable Banach space
(see \cite[Remark~3.12]{Ambrosio2000}), 
it follows that 
$U$ is also the dual to a separable Banach space.
\paragraph{Linear operator}
Let $1 \leq q \leq \infty$,
and let $\iota_q \colon \mathrm{BV}(0,T) \to L^q(0,T)$
be the embedding operator. 
For $q=1$, $\iota_q$ is
sequentially weakly$^*$-to-strongly continuous
(see  \cite[Proposition~3.13]{Ambrosio2000}).
Moreover, 
$\iota_q$ 
is continuous
for $q = \infty$ 
and
compact
for $q < \infty$
(see \cite[Corollary~3.49]{Ambrosio2000}).
Hence,
$\iota_2$ is 
sequentially weakly$^*$-to-strongly continuous.
It follows that 
$B \coloneqq (\iota_2, \ldots, \iota_2)$
 is sequentially weakly$^*$-to-strongly continuous.

\paragraph{Objective}
Leveraging the continuity of the embedding
$W(0,T) \embedding \mathcal{C}([0,T]; \mathbb{R}^n)$, we can define
$\mathscr{J} \colon L^2(0,T)^m \times \Xi \to \mathbb{R}$
by
$\mathscr{J}(w,\xi)
\coloneqq
\ell([\mathscr{S}(w,\xi)](T))
$.
Since $\mathscr{S}$ is Carath\'eodory, 
$\mathscr{J}$ is Carath\'eodory as well.

\paragraph{Constraints}
We define 
$\mathscr{G} \colon L^2(0,T)^m \times \Xi \to \mathcal{C}([0,T];\mathbb{R}^n)$
by
$[\mathscr{G}(w,\xi)](t)
\coloneqq
h([\mathscr{S}(w,\xi)](t))
$,
and
$K \coloneqq \{y \in R \mid y(t) \leq 0 \quad \text{for all} \quad t \in [0,T]\}$.
Since $\mathscr{S}$ is Carath\'eodory,  $\mathscr{G}$ is Carath\'eodory.

\paragraph{Control regularizer}
Since the  set $\Uad$ is bounded and closed, 
and the embedding $\mathrm{BV}(0,T) \embedding L^1(0,T)$
is compact,
$\Uad$ is sequentially weakly$^*$  compact
(cf.\ \cite[Proposition~3.13]{Ambrosio2000}).

\paragraph{Feasibility and finite objective value}
For $u=0$,
our assumption
$f(y_0, 0, \xi) = 0$ for all $\xi \in \Xi$
ensures that $y(t,\xi) = y_0$ solves the initial value
problem, and we have $h(y_0) \leq 0$.
Combined with the fact that $u=0$ is an element of $\Uad$,
we find that $u = 0$ is a feasible point.
Also, $\mathscr{J}(0,\xi)
= \ell(y_0)$. Hence, $F(0) = \ell(y_0) < \infty$.

\subsection{Semilinear elliptic PDE-constrained optimization}
\label{subsect:semilinear-elliptic-pde}

In this example, we consider a risk-neutral tracking-type objective with a semilinear elliptic PDE constraint under uncertainty,
extending \Cref{example:semilinear}.
Semilinear elliptic PDEs
are a fundamental model class for nonlinear PDEs and frequently
arise in PDE-constrained optimization. Here, uncertainty enters through an
unknown diffusion coefficient. Pointwise state constraints are natural in this setting, for instance
when the state represents a physical quantity that must remain below a
prescribed threshold for each diffusion coefficient's realization.
This type of problem was studied in \cite{Kouri2020} without state constraints. Our handling of state constraints will rely on the analysis for the deterministic counterpart from \cite{Haller2009}. For the purposes of illustration, we will focus on a specific nonlinearity and objective function. In this example, we verify \Cref{assumption:general-assumptions} as well as \Cref{assumption:continuous-setting}.

The state is controlled on a bounded Lipschitz domain $\domain \subset \real^d$, $d\in \{2,3\}$. We assume that $\alpha >0$, $y_d \in L^2(D)$, and $y_{\max} \in H^1(D)\cap \mathcal{C}(\bar D)$ with $y_{\max}>0$. Let $\psi(u)\coloneqq I_{\Uad}(u)+(\alpha/2)\lVert u\rVert_{L^2(D)}^2$, where $\Uad\coloneqq\{ u \in L^2(D) \mid u(x) \in [\mathfrak{l}, \mathfrak{u}]  \text{ a.e.~in $D$}\}$ with $\mathfrak{l}$, $\mathfrak{u} \in \mathbb{R}$ and
 $\mathfrak{l}\leq 0\leq \mathfrak{u}$, $\mathfrak{l} \neq \mathfrak{u}$. Consider the problem
\begin{align*}
  \min_{u \in L^2(D)}\,
    (1/2)\mathbb{E}[\|y-y_d\|_{L^2(\domain)}^2]
    +\psi(u)
\end{align*}
with $y\colon D \times \Xi \rightarrow \R$ satisfying the following constraints for all $\xi \in \Xi$:
\begin{align}
  -\nabla \cdot ( \kappa(\xi) \nabla y) + y^3 = u \quad \text{a.e.~$x\in D$}, \quad
  y = 0 \quad \text{on~$\partial D$},& \label{eq:semilinear-strong-formulation}\\
   y \leq y_{\max} \quad \text{for all $x \in D$}&. \nonumber
\end{align}
We assume that for all $\xi$ in a compact set $\Xi \subset \R^m$, the coefficient $\kappa(\xi) \in L^\infty( \domain)$ satisfies $0 < \kappa_{\min}\leq\kappa(\xi)  \leq \kappa_{\max}<\infty$ and  $\xi \mapsto \kappa(\xi)$ is continuous. 

To place this problem in our theoretical framework, we utilize advanced elliptic regularity theory from \cite{Groeger1989,Haller2009}. 
Given a realization $\xi \in \Xi$ and $u \in L^s(D),$ $s>d/2$, a function $y \in W_0^{1,p}(D)\cap L^\infty(D)$ with regularity $p$ (to be specified) is said to be a solution of \eqref{eq:semilinear-strong-formulation} if it fulfills the operator equation
\begin{equation}
\label{eq:semilinear-operator-form}
-\nabla \cdot (\kappa(\xi)\nabla y)+b(y) = w \quad \text{in } W^{-1,p}(D),
\end{equation}
where $w \in W^{-1,p}(D)$ and $b\colon L^\infty(D) \rightarrow W^{-1,p}(D)$ are defined by 
\begin{equation}
\label{eq:identification-u-by}
\langle w,v\rangle \coloneqq \int_D u(x)v(x)\D x, \quad \langle b(y),v\rangle \coloneqq \int_D y^3(x)v(x)\D x, \quad v \in W_0^{1,p}(D).
\end{equation}

The following is shown in \Cref{appendix}. We note that this regularity result can be shown in more generality than demonstrated here; in particular, the result also holds for the case $d=4$ and with more general nonlinear terms and mixed boundary conditions, see \cite[Theorem 6.6]{Haller2009}. There are several aspects at play: on the one hand, we would like to be able to define a compact operator $B$ based on the right-hand side of \eqref{eq:semilinear-operator-form}, but the right-hand side should have enough regularity so that the solution is  continuous. 
\begin{lemma}
\label{lemma:regularity-semilinear}
Under the hypotheses of this section, for every $u \in L^s(D)$ with $s > d/2$ and $\xi \in \Xi$, there exists a unique solution $y\in H_0^1(D)\cap L^\infty(D)$ to
\eqref{eq:semilinear-strong-formulation}. Moreover, the solution belongs to $\mathcal{C}(\bar{D})$  and $\xi \mapsto y(\xi) \in H_0^1(D)$ is continuous on $\Xi$.
\end{lemma}

\Cref{lemma:regularity-semilinear} justifies the definition of the (parametrized) control-to-state operator $\mathscr{S}\colon H^{-1}(D)\times \Xi \rightarrow H_0^1(D)$, $(w,\xi)\mapsto y(\xi)$
solving the equation \eqref{eq:semilinear-operator-form} with $p=2$,
where, with a slight abuse of notation, \( w \) denotes a general element of \( H^{-1}(D) \), not necessarily the specific one appearing in \eqref{eq:identification-u-by}.
We note that based on the definition of $B$ given below and \Cref{lemma:regularity-semilinear}, $\mathscr{S}(Bu,\xi) \in H_0^1(D)\cap\mathcal{C}(\bar{D})$ for $u \in L^2(D)$ and $\xi \in \Xi$.

We now verify \Cref{assumption:general-assumptions}.

\paragraph{Function spaces} We define $U\coloneqq L^2(D)$, $W\coloneqq H^{-1}(D)$, $R\coloneqq H^1(D)\cap \mathcal{C}(\bar{D})$, and $R_0\coloneqq H^1(D)$.
\paragraph{Linear operator} The operator $B$ is defined as the  embedding of $L^2(D)$ into $H^{-1}(D)$. This operator is clearly compact since its adjoint operator $B^* \colon H_0^1(D) \rightarrow L^2(D)$ is compact; see \cite[Theorem 10.9]{Alt2016}. 
\paragraph{Objective} The objective is defined by  \[\mathscr{J}\colon H^{-1}(D)\times\Xi\rightarrow \R, (w,\xi)\mapsto  (1/2)\|\mathscr{S}(w,\xi)-y_d\|_{L^2(\domain)}^2.\]
\paragraph{Constraints} We can define the constraint function by  \[\mathscr{G}\colon H^{-1}(D)\times \Xi \rightarrow H^1(D), (w,\xi) \mapsto \mathscr{S}(w,\xi) - y_{\max}.\]  
We define the set $K \coloneqq \{ r\in R \mid r(x)\leq 0 \text{ for all $x \in \bar{D}$}\}$,
which is a closed cone with a nonempty interior. 
\paragraph{Control regularizer} The set $\Uad$ is nonempty, closed, convex, and bounded.
\paragraph{Feasibility and finite objective value} For $u=0$, $y=0$ is the solution to \eqref{eq:semilinear-strong-formulation}. From the assumptions on $\Uad$ and $y_{\max}$, $u=0$ is  feasible with finite objective value by the regularity assumption on $y_d$.

Under the hypotheses of this section, $\mathscr{S}$ is continuously differentiable with respect to its first argument.
Indeed, the linearized equation
\begin{equation}
\label{eq:semilinear-adjoint}
-\nabla \cdot (\kappa(\xi)\nabla v)+\tilde{b}'(\tilde{y})v = h
\end{equation}
has a unique solution $v \in H_0^1(D)$ for any $h \in H^{-1}(D)$ and $\tilde{y}\in L^\infty(D)$; see \cite[Theorem 6.10]{Haller2009}. Here, $\tilde{b}'(\tilde{y})\colon H_0^1(D)\rightarrow H^{-1}(D)$ is defined such that
\[
\langle \tilde{b}'(\tilde{y})v,w\rangle \coloneqq\int_D 3\tilde{y}^2(x) v(x)w(x) \D x, \quad w \in H_0^1(D).
\]
With the implicit function theorem, $H^{-1}(D)\ni w \mapsto \mathscr{S}(w,\xi) \in H_0^1(D)$ is continuously differentiable, where for fixed $h \in H^{-1}(D)$, the function $v = D_w \mathscr{S}(w,\xi)h \in H_0^1(D)$ solves \eqref{eq:semilinear-adjoint} with $\tilde{y}=\mathscr{S}(w,\xi)$.

Now, we turn to verifying \Cref{assumption:continuous-setting}. 
\paragraph{Sample space}
The sample space $\Xi \subset \R^m$ is compact by assumption.
\paragraph{Smooth objective} The composition of the function $L^2(D) \ni v\mapsto (1/2)\lVert v-y_d\rVert_{L^2(D)}^2$ with $H^{-1}(D)\ni w\mapsto \mathscr{S}(w,\xi)\in H_0^1(D)$ is continuously differentiable, from which we obtain the continuous differentiability of  $\mathscr{J}(\cdot,\xi)$. The other properties in \Cref{assumption:continuous-setting}\ref{itm:j-smooth} can be deduced from the properties of $\mathscr{S}$.
\paragraph{Smooth constraint} Evidently, the continuous differentiability of $\mathscr{G}$ follows by  that of $\mathscr{S}$. The map $\mathscr{G}$ is continuous with respect to its second argument since $\mathscr{S}$ has this property  by \Cref{lemma:regularity-semilinear}. Continuity of $\xi \mapsto D_w \mathscr{S}(w,\xi)$ for every $w \in H^{-1}(D)$  can be argued as in \Cref{lemma:regularity-semilinear}. 

Now, we will turn to applying the optimality conditions in \Cref{lemma:KKT-continuous}. To this end, we suppose that the constraint qualification in \Cref{lemma:KKT-continuous} can be verified in a local optimum $u^*$. We have $DF(u^*) = \E[ B^* D_{w}\mathscr{S}(Bu^*,\xi)^*(\mathscr{S}(Bu^*,\xi)-y_d)]$ and $D\mathcal{G}(u^*)  = D_w \mathscr{S}(Bu^*,\cdot)B$ (see \Cref{rem:assumptions-KKT}). Recalling the definition of the normal cone $N_{\Uad}(u)$ and $\psi(u)=(\alpha/2)\lVert u\rVert_{L^2(D)}^2  + I_{\Uad}(u)$, we have $\partial \psi(u) = \alpha u + N_{\Uad}(u)$.
Therefore, the conditions \eqref{eq:KKT-abstract-robust} translate to
\begin{align}
\label{eq:semilinear-OCs-abstract}
\nonumber
0 \in \E[ B^* D_{w}\mathscr{S}(Bu^*,\xi)^*(\mathscr{S}(Bu^*,\xi)-y_d)]+ \alpha u^*+N_{\Uad}(u^*)+B^* D_w \mathscr{S}(Bu^*,\cdot)^* \lambda^* ,\\
\quad \lambda^* \in \mathcal{K}^-, \quad \langle \lambda^*,\mathcal{G}(u^*)\rangle=0.
\end{align}
Here, $\mathcal{K}\coloneqq\{r \in \mathcal{C}(\Xi;R) \mid r(\xi)(x) \leq 0 \text{ for all $(x,\xi) \in \bar{D}\times \Xi$} \}$.

\section*{Acknowledgments}
We thank the reviewers for their helpful comments and suggestions, which
significantly improved our presentation.%

\appendix
\section{Appendix}
\label{appendix}
\begin{proof}[Proof of \Cref{lemma:regularity-semilinear}]
Let $\xi\in \Xi$ be fixed. We follow arguments used in \cite{Haller2009}. The existence and uniqueness of $y(\xi)\in H_0^1(D)\cap L^\infty(D)$ follows from \cite[Theorem 6.6]{Haller2009} (note that the control in their example appears on the boundary, but the result follows without difficulty for the case of a distributed control). Furthermore, \eqref{eq:semilinear-operator-form} can be equivalently written in the form
\[(-\nabla \cdot \kappa(\xi)\nabla +1)y = g,
\]
with $g\coloneqq w-{b}(y)+\tilde{y}$, where $\tilde{y} \in W^{-1,p}(D)$ is defined by 
\[
\langle \tilde{y},v\rangle= \int_D y(x)v(x)\D x, \quad v \in W_0^{1,p}(D)
\]
and $w$, $b(y)$ are defined according to \eqref{eq:identification-u-by}.
Since $u \in L^s(D)$ with $s>d/2$ and $y \in L^\infty(D)$, it follows by Sobolev embedding theorems that $g \in W^{-1,p}(D)$ for a $p>d$. From \cite[Proposition 4.1]{Haller2009}, there exists some $p_0 >0$ such that $-\nabla \cdot \kappa(\xi)\nabla +1$ is a topological isomorphism between $W_0^{1,p}(D)$ and $W^{-1,p}(D)$ for all $p\in (2-p_0,2+p_0)$. Additionally, $(-\nabla \cdot\kappa(\xi)\nabla +1)^{-1}$ maps $W^{-1,p}(D)$ continuously into $\mathcal{C}(\bar{D})$ by \cite[Theorem 3.3]{Haller2009}. This proves the claimed state regularity  for a fixed realization $\xi$.

To show continuity with respect to $\Xi$, let $u$ be a fixed control and consider $S\colon \Xi \rightarrow H_0^1(D)$ mapping a realization $\xi$ to the solution of \eqref{eq:semilinear-strong-formulation}. We define $y_k\coloneqq S(\xi_k)$ and $y\coloneqq S(\xi)$ with $\xi_k \rightarrow \xi$ as $k \rightarrow \infty$. From the weak formulation of \eqref{eq:semilinear-strong-formulation}, we have
\begin{align*}
  0 &= \int_D \big(\kappa(\xi_k)\nabla y_k -\kappa(\xi)\nabla y\big)\cdot \nabla (y_k -y)+(y_k^3-y^3)(y_k-y)\du x\\
  &\geq \int_D \Big(\kappa(\xi_k)\nabla (y_k-y) +\big(\kappa(\xi_k)-\kappa(\xi)\big) \nabla y \Big) \cdot  \nabla (y_k -y) \du x \\
&\geq c \lVert y_k-y\rVert_{H^1(D)}^2+ \int_D \big(\kappa(\xi_k)-\kappa(\xi)\big) \nabla y  \cdot  \nabla (y_k -y) \du x,
\end{align*}
where $c\coloneqq(1+C_p^2)^{-1}\kappa_{\min}$ and $C_p>0$ is the Poincar\'e--Friedrichs  constant.
Rearranging and applying H\"older's inequality, we have
\begin{align*}
     c\lVert y_k-y\rVert_{H^1(D)}^2 
     &\leq \lVert \big(\kappa(\xi_k)-\kappa(\xi)\big)\nabla y\rVert_{L^2(D)} \lVert y_k - y\rVert_{H^1(D)}
\end{align*}
meaning $c\lVert y_k-y\rVert_{H^1(D)} \leq \mathrm{ess\,sup}_{x\in D} |\kappa(\xi_k)(x)-\kappa(\xi)(x)|\lVert \nabla y\rVert_{L^2(D)}$ and so $y_k \rightarrow y$ in $H^{1}_0(D)$ as $\xi_k\rightarrow \xi$ by continuity of $\xi\mapsto \kappa(\xi)$.
\end{proof}

\subsection*{Acknowledgement}
The first author acknowledges support from the Fondation Math\'ematique Jacques Hadamard [Program Gaspard Monge in Optimization and Operations Research]. The second author was supported by the National Science Foundation
under grant DMS-2410944.
\bibliography{lit}
\end{document}